\documentclass[11pt]{amsart}
\usepackage{mathrsfs}
\usepackage{geometry}
\usepackage{graphicx}
\usepackage{amssymb}
\usepackage{epstopdf}
\usepackage{amsmath,amscd}
\usepackage{amsthm}
\usepackage{enumerate}
\usepackage{url,verbatim}
\usepackage{subfig}
\usepackage{enumitem}
\usepackage{tikz}
\usepackage{enumerate}

\usepackage{fixme}

\calclayout

\makeatletter
\newcommand\xleftrightarrow[2][]{%
  \ext@arrow 9999{\longleftrightarrowfill@}{#1}{#2}}
\newcommand\longleftrightarrowfill@{%
  \arrowfill@\leftarrow\relbar\rightarrow}
\makeatother

\RequirePackage[colorlinks,citecolor=blue,urlcolor=blue]{hyperref}
\usepackage{breakurl}
\theoremstyle{plain}
\newtheorem{theorem}{Theorem}
\newtheorem{definition}[theorem]{Definition}
\newtheorem{lemma}[theorem]{Lemma}
\newtheorem{proposition}[theorem]{Proposition}
\newtheorem{corollary}[theorem]{Corollary}

\newtheorem{remark}[theorem]{Remark}

\newtheorem{conjecture}[theorem]{Conjecture}

\newcommand\RR{{\mathbb R}}
\newcommand\ZZ{{\mathbb Z}}

\newcommand\PP{{\mathbb P}}
\newcommand\HH{{\mathbb H}}

\newcommand\pcs{{p_c^{site}}}

\newcommand\lra{\leftrightarrow}

\renewcommand\ell{l}

\newcounter{mycount}

\numberwithin{equation}{section}
\numberwithin{theorem}{section}
\numberwithin{figure}{section}

\title{Tree embeddings and nonuniqueness in site percolation}

\author{Zhongyang Li}
\address{Department of Mathematics,
University of Connecticut,
Storrs, Connecticut 06269-3009, USA}
\email{zhongyang.li@uconn.edu}
\urladdr{\url{https://mathzhongyangli.wordpress.com}}

\usepackage{xcolor}
\hypersetup{linkcolor=black}
\newenvironment{revision}{\begingroup}{\endgroup}
\newcommand{\rev}[1]{#1}
\newcommand{\thisrevision}[1]{#1}
\newenvironment{thisrevisionblock}{\begingroup}{\endgroup}
\newif\ifmainrevisionred
\mainrevisionredfalse
\newcommand{\mainrevisionstyle}{\ifmainrevisionred\color{red}\hypersetup{allcolors=red}\fi}
\newcommand{\mainrevision}[1]{{\mainrevisionstyle #1}}
\newenvironment{mainrevisionblock}{\begingroup\mainrevisionstyle}{\endgroup}

\begin{document}

\begin{abstract}
Let $G$ be an infinite connected simple graph of minimum degree
at least seven. Suppose $G$ is embedded in the plane so that every
compact set meets only finitely many vertices and edges. We prove that
$p_c^{\mathrm{site}}(G)<1/2$ and that, for every
$p_c^{\mathrm{site}}(G)<p<1-p_c^{\mathrm{site}}(G)$,
Bernoulli$(p)$ site percolation almost surely has infinitely many
infinite open clusters. We also construct a subtree of $G$ with critical
site percolation probability $\sqrt{2}-1$.

We prove a general stability theorem for infinite connected locally
finite graphs $G\subseteq H$ on the same vertex set. Suppose that, at
parameter $a\in(0,1)$, the probability that a deterministic finite set
meets an infinite open cluster in $G$ tends uniformly to one as its
size grows. Then, for each fixed $b\in(a,1]$, under the standard
monotone coupling, almost surely every infinite open cluster in $H$
at parameter $b$ contains an infinite open cluster in $G$ at parameter
$a$. No planarity, transitivity, or uniform degree bound is required.

Finally, we obtain a lower bound on two-point connectivity under
uniqueness for arbitrary infinite locally finite graphs, yielding
a general criterion for nonuniqueness.
\end{abstract}

\maketitle

\begin{revision}
\section{Introduction}

\subsection{Background and the planar theorem}

Let $G=(V,E)$ be an infinite connected locally finite simple graph.
In Bernoulli site percolation with parameter $p\in[0,1]$, each
vertex is independently declared open with probability $p$ and
closed otherwise. An \emph{open cluster} is a connected component
of the subgraph induced by the open vertices; closed clusters are
defined analogously. A cluster is \emph{infinite} if it contains
infinitely many vertices. We write $\PP_p$ for the resulting
probability measure.

A central question is whether an infinite open cluster, when it
exists, is unique. A graph is \emph{quasi-transitive} if its
automorphism group has finitely many vertex orbits. On an amenable
quasi-transitive graph, the Burton--Keane argument shows that, for
each fixed $p$, there is almost surely at most one infinite open
cluster~\cite{BK89,Kesten02}. On non-amenable graphs, the existence
of a phase with infinitely many infinite open clusters is a basic
problem raised by Benjamini and Schramm~\cite{bs96}. For planar
graphs, duality and matching-graph methods have led to substantial
progress under symmetry assumptions~\cite{bsjams,LP16,ZL17,GrZL22,GrZL221}.
Here we study nonuniqueness without assuming quasi-transitivity.

Write $p_c^{\mathrm{site}}(G)$ for the infimum of the parameters
at which an infinite open cluster exists almost surely, and
$p_u^{\mathrm{site}}(G)$ for the infimum of the parameters at which
the probability of a unique infinite open cluster is positive.
Our main application concerns graphs properly embedded in
$\RR^2$: every compact subset of the plane meets only finitely
many vertices and edges.

Benjamini and Schramm~\cite[Conjecture~7]{bs96} proposed the following
relation between a planar degree condition and nonuniqueness.

\end{revision}
\begin{conjecture}\label{c11}
Let $G$ be an infinite connected planar graph with minimum degree at least $7$. Then
\[
p_c^{\mathrm{site}}(G)<\tfrac12,
\]
and for every
\[
p\in \bigl(p_c^{\mathrm{site}}(G),\,1-p_c^{\mathrm{site}}(G)\bigr),
\]
Bernoulli site percolation on $G$ has almost surely infinitely many infinite open clusters.
\end{conjecture}
\begin{revision}

The distinction between planarity and proper embeddability matters:
the unrestricted conjecture has counterexamples~\cite{Li26Ends}.
We prove its conclusion in the properly embedded setting.

\begin{theorem}\label{mt1}
Let $G$ be an infinite connected simple graph properly embedded in $\RR^2$, and assume that every vertex of $G$ has degree at least $7$. Then
\[
p_c^{\mathrm{site}}(G)<\tfrac12.
\]
Moreover, for every
\[
p\in \bigl(p_c^{\mathrm{site}}(G),\,1-p_c^{\mathrm{site}}(G)\bigr),
\]
i.i.d.\ Bernoulli$(p)$ site percolation on $G$ has almost surely infinitely many infinite open clusters.
\end{theorem}

There is no quasi-transitivity assumption or uniform upper bound on
vertex degrees or finite face degrees, and infinite faces are allowed.
Alongside this planar theorem, we develop connectivity and stability
results for Bernoulli site percolation on general locally finite
graphs.

\subsection{Relation to previous work}

Haslegrave and Panagiotis~\cite{HP19} proved
$p_c^{\mathrm{site}}(G)<1/2$ for the graphs in
Theorem~\ref{mt1}, with a stronger quantitative bound and a lower
bound on the size of the external vertex boundary of every finite
set. We use both estimates in the proof of nonuniqueness. We also
give a separate geometric proof of the strict inequality by
constructing, at any prescribed vertex of $G$, a rooted subtree
$T\subseteq G$ with
\begin{align}
p_c^{\mathrm{site}}(G)\le p_c^{\mathrm{site}}(T)
=\sqrt2-1<\tfrac12.
\label{it1}
\end{align}

After the first version of the present paper was posted on arXiv, Glazman, Harel, and Zelesko~\cite{GHZ25} proved
a $0/\infty$ theorem for a broad class of site-percolation
processes on arbitrary infinite locally finite planar graphs.
Their assumptions are tail triviality, positive association,
and stochastic domination of the open set by the closed set.
For Bernoulli site percolation, the theorem applies when
$p\le1/2$, and therefore gives almost surely infinitely many
infinite open clusters for $p_c^{\mathrm{site}}(G)<p\le1/2$.
It requires neither a proper embedding nor a minimum-degree
assumption. Its proof uses planar boundary decompositions and
an argument counting disjoint arms.

The subsequent work~\cite{Li26Ends} treats the upper part of the
coexistence interval through a condition on the end structure.
When the set of cycle-separation equivalence classes of ends is
countable, its result, combined with the theorem of
Glazman--Harel--Zelesko, gives nonuniqueness throughout
$(p_c^{\mathrm{site}}(G),1-p_c^{\mathrm{site}}(G))$.
Properly embedded planar graphs satisfy that condition, so the
conclusion includes the graphs of Theorem~\ref{mt1}.
That proof uses Freudenthal embeddings, critical parameters
associated with directions toward infinity, and planar
explorations adapted to the ends.

Beyond the explicit tree construction and the planar
nonuniqueness theorem, we establish three results for general
locally finite graphs. We prove that graphs all of whose ends
are thin have infinitely many infinite open clusters throughout
the supercritical interval $(p_c^{\mathrm{site}},1)$
(Theorem~\ref{rev:thin-nonunique}). Under a uniform finite-set
hitting assumption, we show that every infinite component of
$H_b$ contains an infinite component of $G_a$ whenever
$G\subseteq H$ have the same vertex set and $a<b$
(Theorem~\ref{rev:sprinkled-stability}). We also derive a
quantitative lower bound on two-point connectivity under
uniqueness (Proposition~\ref{la68}), yielding the nonuniqueness
criterion in Corollary~\ref{l83}. These results require neither
planarity, quasi-transitivity, nor a uniform bound on vertex
degrees. Our proof of the planar theorem combines these general
tools with induced-cycle estimates, independent thinning, and
finite separation, without invoking the results of
\cite{GHZ25,Li26Ends}.

\subsection{Results for general graphs}

Our connectivity arguments use a quantitative site-percolation estimate
from~\cite{perc24}, based on differential inequalities;
compare~\cite{DCT15}. We first give a quantitative consequence of
uniqueness. Let
$\mathcal A_1$ denote the event that there is a unique infinite
open cluster.

\end{revision}
\begin{proposition}\label{la68}
Let \(G=(V,E)\) be an infinite, connected, locally finite graph. Then for every
\(u,v\in V\),
\[
\PP_p(u\leftrightarrow v)\ge
\PP_p(u\leftrightarrow\infty)\,
\PP_p(v\leftrightarrow\infty)\,
\PP_p(\mathcal A_1).
\]
\end{proposition}
\begin{revision}

In the supercritical phase, infinitely many vertices have their
probability of belonging to an infinite open cluster bounded below
by a common positive constant. Proposition~\ref{la68} therefore
rules out uniqueness when suitable connection probabilities tend to
zero. To state one consequence, let $d_G$ be graph distance and set
\begin{align}
p_{\mathrm{conn}}(G):=
\sup\Bigl\{p\in(0,1):
\lim_{n\to\infty}
\sup_{\substack{x,y\in V\\ d_G(x,y)\ge n}}
\PP_p(x\leftrightarrow y)=0
\Bigr\}.
\label{df87}
\end{align}
The supremum of an empty parameter set is taken to be zero.

\end{revision}
\begin{corollary}\label{l83}
Let \(G=(V,E)\) be an infinite, connected, locally finite graph. If
\[
p_c^{\mathrm{site}}(G)<p<p_{\mathrm{conn}}(G),
\]
then Bernoulli site percolation on \(G\) has almost surely infinitely many
infinite open clusters.
\end{corollary}
\begin{revision}

\begin{mainrevisionblock}
Our second main theorem concerns graphs all of whose ends are thin.
An end is called \emph{thin} if there is a finite bound on the number
of pairwise vertex-disjoint rays representing it;
see~\cite[Section~2.1]{BCP}. The bound on disjoint rays may depend
on the end.

\begin{theorem}[Supercritical percolation on thin-end graphs]
\label{rev:thin-nonunique}
Let $X$ be an infinite connected locally finite graph
all of whose ends are thin. For every $\pcs(X)<s<1$,
Bernoulli site percolation has almost surely infinitely
many infinite open clusters.
\end{theorem}
\end{mainrevisionblock}

Our \mainrevision{third} main theorem compares percolation on two graphs with the
same vertex set. Its hypothesis says that large deterministic finite
sets are likely to meet an infinite component at the lower
parameter, uniformly over the choice of the set.

\begin{mainrevisionblock}
\begin{theorem}[Stability under edge additions and sprinkling]
\label{rev:sprinkled-stability}
Let $G\subseteq H$ be infinite connected locally finite graphs on the
same countable vertex set $V$. Couple site percolation by independent
uniform variables $U(v)\sim\mathrm{Unif}[0,1]$, putting $V_t=\{v:U(v)\le t\}$ and letting
$G_t,H_t$ be the induced graphs on $V_t$. Fix $0<a<b\le1$.
Let $I_a$ denote the union of the vertex sets of the infinite
components of $G_a$. Suppose there is a function $\epsilon(n)\to0$ such that
\begin{equation}\label{rev:uniform-hit-assumption}
 \sup_{\substack{A\subset V\text{ finite}\\|A|=n}}
              \PP(A\cap I_a=\varnothing)\le\epsilon(n).
\end{equation}
Then almost surely every infinite component of $H_b$ contains an infinite
component of $G_a$.
\end{theorem}
\end{mainrevisionblock}

In particular, infinitely many infinite components of $H_b$ imply
infinitely many infinite components of $G_a$. This is the comparison
used to transfer nonuniqueness from a triangulation to the original
graph. Both graphs may have unbounded degrees.

Schonmann~\cite[Theorem~3.1]{Schonmann99} proved stability between
two percolation parameters on a single graph of uniformly bounded
degree under uniform percolation over graph balls.
Theorem~\ref{rev:sprinkled-stability} adapts that argument to two
locally finite graphs and allows unbounded degrees, under the
stronger uniform hitting assumption for arbitrary finite sets of a
given size. The finite-set estimate in
Lemma~\ref{rev:uniform-hitting} verifies this assumption from a
vertex-expansion inequality.

Related work~\cite{Li26Packing} bounds the probability that a set
misses all infinite clusters in terms of a recursive packing number.
Under a vertex-expansion assumption, our estimate gives
uniform control over all finite sets of a given cardinality, as
required by Theorem~\ref{rev:sprinkled-stability}.

\subsection{Proof strategy}

The proof of Theorem~\ref{mt1} uses two overlapping parameter ranges.
Set
\begin{mainrevisionblock}
\begin{equation}\label{eq:a0-q0}
 a_0=\frac{3-\sqrt5}{2},\qquad
 q_0=\frac{2}{3\sqrt3}.
\end{equation}
\end{mainrevisionblock}
The quantitative estimates of Haslegrave and Panagiotis~\cite{HP19}
give $p_c^{\mathrm{site}}(G)<q_0$. Their expansion bound also
allows us to apply the stability theorem whenever its lower
parameter is greater than $a_0$. We triangulate the finite faces
of $G$ by adding edges, obtaining a properly embedded locally
finite graph $H$ on the same vertices.

For $p_c^{\mathrm{site}}(G)<p<q_0$, an estimate on induced cycles
shows that every component of a slightly denser configuration on
$G$ has only thin ends. Independent thinning of open vertices then
produces the configuration at parameter $p$. An auxiliary
supercritical parameter below $p$ ensures that the thinning is
strictly supercritical on a suitable component, where we apply the
general thin-end theorem. This gives infinitely many infinite
components of $G_p$.

For $1-q_0<b<1-p_c^{\mathrm{site}}(H)$, the cycle estimate and a
thinning argument show that Bernoulli$(b)$ site percolation on $H$
almost surely has an infinite closed cluster with infinitely many
ends. The tree structure of the adjacency graph of open and closed clusters,
together with planar separation, shows that if there were only
finitely many infinite open clusters, a finite vertex deletion
would leave an infinite component containing no infinite open
cluster. This contradicts the percolation estimate that guarantees an infinite open cluster in
every infinite component after such a deletion.
Since $p_c^{\mathrm{site}}(H)\le p_c^{\mathrm{site}}(G)$, for each
$a_0<p<1-p_c^{\mathrm{site}}(G)$ we can choose such a $b>p$.
The stability theorem then transfers nonuniqueness from $H_b$ to
$G_p$.

Since $a_0<q_0$, the two ranges overlap and cover the entire
interval in Theorem~\ref{mt1}.

\subsection{Organization of the paper}

Section~\ref{sbn} fixes notation.
Section~\ref{sect:tree} gives the embedded-tree construction.
Sections~\ref{sect:ccp} and~\ref{sect:planarg} develop the general
connectivity and thin-end results and prove the lower parameter
range. Sections~\ref{sect:5} and~\ref{sect:6} supply planar
separation, the cycle estimate, and finite-face triangulation.
Section~\ref{sect:up} establishes the uniform hitting estimate,
proves the stability theorem, and completes the proof of
Theorem~\ref{mt1}.
\end{revision}

\section{Background and Notation}\label{sbn}

\begin{revision}
Throughout, graphs are simple and locally finite, with countable
vertex sets. For a graph $G=(V,E)$, let $V$ and $E$ denote its
vertex set and edge set, respectively. For $u,v\in V$, write
$u\sim_G v$ if $u$ and $v$ are adjacent in $G$.

For site percolation on a graph $X$ with vertex set $V$, let
$(U(v))_{v\in V}$ be independent random variables, each uniformly
distributed on $[0,1]$. For $t\in[0,1]$, define
\[
V_t=\{v\in V:U(v)\le t\},
\qquad X_t=X[V_t].
\]
Thus $X_t$ is the subgraph induced by the $t$-open vertices,
and its connected components are the $t$-open clusters.
Unless otherwise stated, configurations at different parameters
are coupled using these same variables. When considering graphs
$G\subseteq H$ on the same vertex set, we also use the same
variables for both graphs, so that
\[
G_t=G[V_t]\subseteq H[V_t]=H_t.
\]

For $A\subseteq V$, define its external vertex boundary by
\[
\partial_G A
=\{z\in V\setminus A:\text{ there exists }v\in A
  \text{ such that }z\sim_G v\}.
\]
For $v\in V$, write $N_G(v)=\partial_G\{v\}$ for its neighbour set.

A proper plane embedding is one for which every compact subset
of the plane meets only finitely many vertices and edges.

A finite walk in $G$ is an alternating sequence
\[
(v_0,e_0,v_1,\ldots,e_{n-1},v_n),
\]
where each $v_i\in V$ is a vertex and each
$e_i=\langle v_i,v_{i+1}\rangle\in E$ is the edge joining
$v_i$ to $v_{i+1}$. One-sided and two-sided infinite walks are
defined analogously. Since $G$ is simple, we usually specify
a walk by its vertex sequence.
\end{revision}

A walk is \emph{closed} if it can be written as $(v_0,v_1,\ldots,v_n)$ with $v_0=v_n$.
A walk is \emph{self-avoiding} if it visits no vertex more than once.
\thisrevision{A \emph{cycle} (or \emph{simple cycle}) is a closed walk
$(v_0,v_1,\ldots,v_m)$ of length $m\ge3$ such that $v_m=v_0$ and
$v_0,\ldots,v_{m-1}$ are pairwise distinct.}

Once $G$ is embedded in $\RR^2$, a \emph{face} is a maximal connected component of $\RR^2\setminus G$.
Faces may be bounded or unbounded. A face is \emph{finite} if it is bounded and its boundary consists
of finitely many edges.

A \emph{ray} in a graph is a one-sided infinite self-avoiding path. If the graph is
embedded in \(\mathbb R^2\), we call a ray \emph{proper} if its embedded realization
eventually leaves every compact subset of \(\mathbb R^2\). In a properly embedded
locally finite graph, every ray is proper.

\begin{definition}\label{df64}
Let $G=(V,E)$ be an infinite, connected, locally finite, simple, planar graph and fix an embedding of
$G$ into the plane. The \textbf{matching graph} $G_*=(V,E_*)$ has the same vertex set as $G$, and for
distinct $u,v\in V$ we declare $\langle u,v\rangle\in E_*$ if and only if
\begin{enumerate}
\item $\langle u,v\rangle\in E$; or
\item $\langle u,v\rangle\notin E$, but $u$ and $v$ share a finite face in $G$.
\end{enumerate}
\end{definition}

\begin{definition}[Ends]
\label{def:ends}
Let \(G=(V,E)\) be an infinite connected graph. For a vertex set \(K\subset V\), write
\[
G\setminus K:=G[V\setminus K]
\]
for the induced subgraph obtained by deleting the vertices in \(K\) and all
edges incident to them.

An \emph{end} of \(G\) is a map
\(\psi\) assigning to every finite vertex set \(K\subset V\) an infinite
connected component \(\psi(K)\) of \(G\setminus K\), such that
\[
\psi(K)\subseteq \psi(K')
\qquad\text{whenever }K'\subseteq K .
\]
\begin{thisrevisionblock}
The number of ends is
\[
\sup_{K\subseteq V,\ |K|<\infty}
\#\{\text{infinite connected components of }G\setminus K\}.
\]
\end{thisrevisionblock}
\end{definition}

Assume $G$ is connected and properly embedded in the plane. The boundary of a finite face is a closed
walk; its \emph{degree} $|f|$ is the number of steps of that walk. (A vertex or an edge may be visited
multiple times by the boundary walk, so $|f|$ may exceed the number of distinct boundary vertices.)

If $v$ is a vertex and $f$ is a face, we
write $v\sim f$ when $v$ lies on the boundary of $f$.

For $\omega\in\Omega$, we write $u \lra v$ if there exists an open path in $G$ with endpoints $u$ and $v$,
and $x\xleftrightarrow{A} v$ if such a path exists using only vertices in $A\subseteq V$. A similar
notation is used for the existence of infinite open paths. When the corresponding open paths exist in
the matching graph $G_*$, we use the notation $\xleftrightarrow{\ast}$.

For $i\in\{0,1\}$, write
$x\xleftrightarrow{\,i\,}y$ if there is a path from $x$ to $y$
all of whose vertices have state $i$, where $1$ denotes open and
$0$ denotes closed. Similarly,
$x\xleftrightarrow{\,i\,}\infty$ means that an infinite
self-avoiding path of vertices in state $i$ starts at $x$.
When a subgraph is specified, the path is required to lie in
that subgraph.

\section{Embedded Trees}\label{sect:tree}

The goal of this section is to show that 
every infinite, connected \rev{simple} graph properly embedded into the plane with vertex degree at least 7 
possesses a tree as a subgraph in which every vertex other than the root vertex has degree $3$ or $4$. We start with geometric properties of planar graphs.

\begin{definition}\label{df22}
Let $G=(V,E)$ be a locally finite plane graph. At a vertex $v$ of
degree $d\ge1$, list the incident edges $e_1,\ldots,e_d$ in
counterclockwise order around $v$, and put $e_{d+1}=e_1$.
Let $\mathcal S(v)=\{\sigma_1,\ldots,\sigma_d\}$ be the sectors swept
counterclockwise from $e_i$ to $e_{i+1}$, and let $f(\sigma_i)$
be the face occupying $\sigma_i$. If $d=1$, its sole sector is
occupied by the unique face incident to $v$. If $d=0$, put
$\mathcal S(v)=\varnothing$.
The same face may occupy several sectors, which are counted separately.
For an infinite face $f$, use the convention $(|f|-2)/|f|=1$.
Define the curvature at $v$ by
\[
\kappa(v)=2\pi-
\sum_{\sigma\in\mathcal S(v)}
\frac{|f(\sigma)|-2}{|f(\sigma)|}\pi.
\]
\end{definition}

\begin{lemma}[Combinatorial Gauss--Bonnet for a cycle]\label{lem24}
Let $G=(V,E)$ be a connected locally finite graph properly
embedded in $\mathbb R^2$.
Let $C$ be a simple cycle with $n$ vertices, and let $F_C$ be
the set of finite faces contained in the bounded component of
$\mathbb R^2\setminus C$.
Let $V_C^\circ$ be the set of vertices strictly inside $C$.
Then
\begin{equation}\label{eq:GB-comb}
\sum_{z\in V(C)}
\sum_{\sigma\in\mathcal S_C(z)}
\frac{|f(\sigma)|-2}{|f(\sigma)|}\pi
=
(n-2)\pi+\sum_{z\in V_C^\circ}\kappa(z).
\end{equation}
Here, for each vertex $z$ on or inside $C$, the set
$\mathcal S_C(z)\subseteq\mathcal S(z)$ consists of the sectors
at $z$ lying inside $C$.

In particular, if $\kappa(z)\le 0$ for all $z\in V$, then
\begin{equation}\label{eq:GB-ineq}
\sum_{z\in V(C)}
\sum_{\sigma\in\mathcal S_C(z)}
\frac{|f(\sigma)|-2}{|f(\sigma)|}\pi
\le (n-2)\pi.
\end{equation}
\end{lemma}

\begin{proof}
Let $m=|F_C|$ and $s=|V_C^\circ|$, and let $t$ be the number
of edges whose relative interiors lie in the open disk bounded
by $C$.
Each of these $t$ edges contributes its two sides, possibly to
the same face, while each edge of $C$ contributes one side.
Thus
\[
\sum_{f\in F_C}|f|=2t+n,
\]
and hence
\[
\sum_{f\in F_C}(|f|-2)\pi=(2t+n-2m)\pi.
\]

Each face $f\in F_C$ occupies exactly $|f|$ sectors at its
boundary vertices. Grouping these sectors by their vertices gives
\begin{align*}
\sum_{f\in F_C}(|f|-2)\pi
&=
\sum_{z\in V_C^\circ}
\sum_{\sigma\in\mathcal S_C(z)}
\frac{|f(\sigma)|-2}{|f(\sigma)|}\pi \\
&\quad+
\sum_{z\in V(C)}
\sum_{\sigma\in\mathcal S_C(z)}
\frac{|f(\sigma)|-2}{|f(\sigma)|}\pi.
\end{align*}

For $z\in V_C^\circ$, every sector at $z$ lies inside $C$.
Therefore, by Definition~\ref{df22},
\[
\sum_{\sigma\in\mathcal S_C(z)}
\frac{|f(\sigma)|-2}{|f(\sigma)|}\pi
=
2\pi-\kappa(z).
\]
Consequently,
\begin{align*}
&\sum_{z\in V(C)}
\sum_{\sigma\in\mathcal S_C(z)}
\frac{|f(\sigma)|-2}{|f(\sigma)|}\pi \\
&\qquad=
(2t+n-2m-2s)\pi
+\sum_{z\in V_C^\circ}\kappa(z).
\end{align*}

Finally, Euler's formula for the planar graph consisting of
the vertices and edges on or inside $C$ gives
\[
(s+n)-(t+n)+m=1,
\]
so $t=s+m-1$.
Substituting this into the preceding identity proves
\eqref{eq:GB-comb}.
The inequality \eqref{eq:GB-ineq} follows immediately when
$\kappa(z)\le 0$ for all $z\in V$.
\end{proof}

\begin{remark}
Recall that the hyperbolic plane $\HH^2$ has constant curvature $-1$. Let $C$ be \thisrevision{a simple geodesic polygon in} $\HH^2$ with vertex set $\{v_1,\ldots,v_n\}_{v_i\in \HH^2}$, and edges $\{\langle v_i,v_{i+1}\rangle\}_{1\leq i\leq n}$ ($v_{n+1}:=v_1$) such that $\langle v_i,v_{i+1}\rangle$ is the geodesic in $\HH^2$ joining $v_i$ and $v_{i+1}$. Let $R_C$ be the bounded region enclosed by $C$. \thisrevision{Take all interior angles on the $R_C$ side of the boundary.} Then by the Gauss-Bonnet formula of $\HH^2$ we have
\begin{align}
\sum_{i\in[n]}\mathrm{internal\ angle\ of}\ C\ \mathrm{at}\ v_i=(n-2)\pi-\mathrm{Area}(R_C)\leq (n-2)\pi.\label{gb}
\end{align}
In general computing the internal angle at each $v_i$ in (\ref{gb}) is not straightforward. However, (\ref{eq:GB-comb}) is expressed in terms of face degrees instead of internal \rev{angles}; and appears to be easier to verify in various situations.
The Gauss-Bonnet formula was used to study random walks on planar graphs; see \cite{ZA97,WW98}.
\end{remark}

\begin{lemma}\label{l48}
Let $G=(V,E)$ be a \rev{connected, locally finite} properly embedded planar graph and assume $\kappa(v)\le 0$ for all $v\in V$.
Let $(v_0,v_1,\dots)$ be a (finite or infinite) \rev{nonbacktracking} walk in $G$.
For each index $n$ for which both $v_{n-1}$ and $v_{n+1}$ are defined, set
$e_n^-:=\langle v_{n-1},v_n\rangle$ and $e_n^+:=\langle v_n,v_{n+1}\rangle$.

Let $\mathcal S_L(n)$ (resp.\ $\mathcal S_R(n)$) be the set of sectors
at $v_n$ swept when turning from $e_n^-$ to $e_n^+$ clockwise
(resp.\ counterclockwise). Assume that for every such $n$,
\begin{equation}\label{gep}
\min\Bigg\{
\sum_{\sigma\in\mathcal S_L(n)}
\frac{|f(\sigma)|-2}{|f(\sigma)|}\pi,\;
\sum_{\sigma\in\mathcal S_R(n)}
\frac{|f(\sigma)|-2}{|f(\sigma)|}\pi
\Bigg\}\ge\pi.
\end{equation}
Then the walk $(v_0,v_1,\dots)$ is self-avoiding.
\end{lemma}

\begin{proof}
Assume the walk is not self-avoiding.
Let $j$ be the smallest index such that $v_j$ has appeared before, and let $i<j$ satisfy $v_i=v_j$.
By minimality of $j$, the vertices $v_0,\dots,v_{j-1}$ are pairwise distinct, hence
\[
C := (v_i,v_{i+1},\dots,v_{j-1},v_j=v_i)
\]
is a simple cycle.

Orient $C$ according to the traversal $v_i\to v_{i+1}\to\cdots\to v_{j-1}\to v_i$.
Fix $k$ with $i<k<j$ and set $z=v_k$.
At time $k$ the walk enters $z$ through $\langle v_{k-1},z\rangle$ and leaves through $\langle z,v_{k+1}\rangle$,
which are precisely the two edges of $C$ incident to $z$.
The sectors in $\mathcal S_C(z)$ coincide with either
$\mathcal S_L(k)$ or $\mathcal S_R(k)$, according to the orientation
of $C$.
Using \eqref{gep} we obtain
\[
\sum_{\sigma\in\mathcal S_C(z)}
\frac{|f(\sigma)|-2}{|f(\sigma)|}\pi \;\ge\; \pi
\qquad\text{for every } z\in V(C)\setminus\{v_i\}.
\]
Summing over all vertices of $C$ except $v_i$ yields
\begin{equation}\label{eq:lower}
\sum_{z\in V(C)\setminus\{v_i\}}
\sum_{\sigma\in\mathcal S_C(z)}
\frac{|f(\sigma)|-2}{|f(\sigma)|}\pi
\;\ge\; (|C|-1)\pi.
\end{equation}

On the other hand, by Lemma~\ref{lem24} (inequality \eqref{eq:GB-ineq}) applied to the cycle $C$ and the assumption $\kappa\le 0$,
\[
\sum_{z\in V(C)}
\sum_{\sigma\in\mathcal S_C(z)}
\frac{|f(\sigma)|-2}{|f(\sigma)|}\pi
\;\le\; (|C|-2)\pi.
\]
Dropping the (nonnegative) contribution of the single vertex $v_i$ preserves the inequality, hence
\[
\sum_{z\in V(C)\setminus\{v_i\}}
\sum_{\sigma\in\mathcal S_C(z)}
\frac{|f(\sigma)|-2}{|f(\sigma)|}\pi
\;\le\; (|C|-2)\pi,
\]
which contradicts \eqref{eq:lower} since $(|C|-1)\pi>(|C|-2)\pi$.
Therefore the walk must be self-avoiding.
\end{proof}

\begin{lemma}\label{lem:kappa-nonpos}
Assume every (finite) face has degree at least $k\ge 3$.
If $\deg(v)\ge d$ and $d\cdot \frac{k-2}{k}\ge 2$, then $\kappa(v)\le 0$.
In particular,
(i) if $\deg(v)\ge 6$ and $|f|\ge 3$ for all faces $f\sim v$, then $\kappa(v)\le 0$;
(ii) if $\deg(v)\ge 4$ and $|f|\ge 4$ for all faces $f\sim v$, then $\kappa(v)\le 0$;
(iii) if $\deg(v)\ge 7$ and $|f|\ge 3$ for all faces $f\sim v$, then $\kappa(v)< 0$.
\end{lemma}

\begin{corollary}\label{l49}Let $G=(V,E)$ be a connected planar graph properly embedded into $\RR^2$ such that each vertex degree is at least 6 and each face degree is at least 3.
For each $v\in V$, label the incident edges of $v$ by 0, 1,\ldots, $\mathrm{deg}(v)-1$ in counterclockwise order. 
Consider the following walk on $G$:
\begin{itemize}
    \item For any $1\leq n\leq \mathrm{length\ of\ the\ walk}-1$, if the walk visits vertex $v_n$ at the $n$th step, and the edge visited immediately before $v_n$ is the edge labelled by $a$, then the edge visited immediately after $v_n$ is the edge labelled by $[(a+3)\mod \mathrm{deg}(\rev{v_n})]$.
\end{itemize}
Then the walk is self-avoiding.

Similarly, the following walk is also self-avoiding:
\begin{itemize}
    \item For any $1\leq n\leq \mathrm{length\ of\ the\ walk}-1$, if the walk visits vertex $v_n$ at the $n$th step, and the edge visited immediately before $v_n$ is the edge labelled by $a$, then the edge visited immediately after $v_n$ is the edge labelled by $[(a-3)\mod \mathrm{deg}(\rev{v_n})]$.
\end{itemize}
\end{corollary}

\begin{corollary}\label{cl27}Let $G=(V,E)$ be a connected planar graph properly embedded into $\RR^2$ such that each vertex degree is at least 4 and each face degree is at least 4. 
For each $v\in V$, label the incident edges of $v$ by 0, 1,\ldots, $\mathrm{deg}(v)-1$ in counterclockwise order.
Consider the following walk on $G$:
\begin{itemize}
    \item For any $1\leq n\leq \mathrm{length\ of\ the\ walk}-1$, if the walk visits vertex $v_n$ at the $n$th step, and the edge visited immediately before $v_n$ is the edge labelled by $a$, then the edge visited immediately after $v_n$ is the edge labelled by $[(a+2)\mod \mathrm{deg}(\rev{v_n})]$.
\end{itemize}
Then the walk is self-avoiding.
\end{corollary}
\begin{thisrevisionblock}
\begin{proof}[Proof of Corollaries~\ref{l49} and~\ref{cl27}]
In Corollary~\ref{l49}, each permitted turn has at least three sectors
on each side, each contributing at least $\pi/3$. In
Corollary~\ref{cl27}, the corresponding bounds are two sectors and
$\pi/2$. In both cases, $\kappa\le0$ by
Lemma~\ref{lem:kappa-nonpos}, so Lemma~\ref{l48} applies.
\end{proof}
\end{thisrevisionblock}

\begin{lemma}\label{lem28}
Let $G=(V,E)$ be an infinite, connected planar graph properly embedded into $\RR^2$
such that the minimal vertex degree is at least $7$.
Then every cycle consisting of $3$ edges bounds a degree-$3$ face.
\end{lemma}

\begin{proof}
Let $C$ be a $3$-cycle with vertices $a,b,c$.
By Lemma~\ref{lem:kappa-nonpos}(iii), we have $\kappa(v)\le 0$ for all $v\in V$.

Assume for contradiction that the bounded region enclosed by $C$ contains a vertex or an edge of $G$.
\thisrevision{Some boundary vertex has an edge entering the open disk. Indeed,
connectedness gives such an edge if there is an interior vertex, and
simplicity excludes an extra edge joining two vertices of the triangle.
That vertex has at least two interior face sectors, possibly belonging
to the same face. Each other boundary vertex has at least one.}

Since every finite face has degree at least $3$, we have $\frac{|f|-2}{|f|}\pi\ge \frac{\pi}{3}$.
Therefore
\[
\sum_{z\in V(C)}
\sum_{\sigma\in\mathcal S_C(z)}
\frac{|f(\sigma)|-2}{|f(\sigma)|}\pi
\;\ge\; \frac{2\pi}{3}+\frac{\pi}{3}+\frac{\pi}{3}
\;=\;\frac{4\pi}{3}\;>\;\pi \;=\; (|C|-2)\pi,
\]
which contradicts \eqref{eq:GB-ineq} applied to $C$.
Hence the bounded region enclosed by $C$ contains no vertices or edges of $G$, so $C$ is the boundary of a face,
and that face has degree $3$.
\end{proof}

In the discussions above, we do not exclude the case that the boundary of a finite face is a closed walk visiting one vertex multiple times. The following lemma shows that nonpositive vertex curvature
rules out repeated vertices on the boundary of a finite face.

\begin{lemma}\label{lc29}
Let $G=(V,E)$ be an infinite, connected planar graph properly embedded into $\RR^2$
such that $\kappa(v)\leq 0$ for all $v\in V$.
Then the boundary of every finite face is a cycle.
\end{lemma}

\begin{proof}
Let $f$ be a finite face and fix $x\in f$. Let $W$ be its
boundary walk, oriented with $f$ on the left. The winding number
of $W$ about $x$ is one, even if $W$ repeats vertices or edges.
Indeed, an edge bordering $f$ on both sides is traversed once
in each direction, so its contributions cancel.

Write $W=(v_0,\ldots,v_m=v_0)$. Choose $i<j$ with
$v_i=v_j$ such that $j-i$ is minimal. Then
$(v_i,\ldots,v_j)$ is either a simple cycle or a
two-edge backtrack. Remove this closed subwalk from $W$
and repeat. Since $W$ is finite, this decomposes it
into simple cycles and backtracks.
Since winding number is additive and backtracks contribute zero,
some simple cycle $C$ consisting of edges of $W$ has nonzero
winding number about $x$. Thus $x$ lies inside $C$.
Since $f$ is connected and disjoint from $C$, the entire face
$f$ lies inside $C$. Moreover, every edge of $C$ belongs to
$W$, so $|f|\ge |C|$.

Suppose $|f|>|C|$. Since $f$ lies inside $C$ and borders
every edge of $C$, each vertex of $C$ has at least one
interior sector occupied by $f$. Hence
\[
\sum_{z\in V(C)}
\sum_{\sigma\in\mathcal S_C(z)}
\frac{|f(\sigma)|-2}{|f(\sigma)|}\pi
\ge |C|\frac{|f|-2}{|f|}\pi
> (|C|-2)\pi,
\]
contradicting \eqref{eq:GB-ineq}.

Therefore $|f|=|C|$. Since $W$ contains every edge of $C$
and has exactly $|C|$ steps, it traverses $C$ once and
contains no other edges. Thus the boundary of $f$ is
the simple cycle $C$.
\end{proof}

Under the assumptions of Lemma \ref{lc29}, the degree $|f|$ of a finite face $f$ is equal to the number of vertices on the boundary of the face which is the same as the number of edges on the boundary of the face. \thisrevision{We now impose minimum degree at least seven to construct the embedded
tree of Lemma~\ref{l59} and derive its percolation consequence.}

\begin{revision}
\begin{lemma}\label{l59}
Let $G=(V,E)$ be an infinite connected simple graph properly embedded
in $\RR^2$, with minimum vertex degree at least seven. At every
$v\in V$ there is an embedded rooted tree $T\subseteq G$ such that:
\begin{enumerate}
\item the root $v$ has degree two, and every other vertex has degree
three or four;
\item for every vertex $\sigma$ and every $n\ge1$, at least one third
of the descendants of $\sigma$ at distance $n$ have degree four;
\item every forward subtree has branching number $1+\sqrt2$ and site
critical probability $\sqrt2-1$.
\end{enumerate}
For each $\sigma\in V(T)$, let $T_\sigma$ denote the
forward subtree consisting of $\sigma$ and all its
descendants, with the inherited tree edges. The degree assertion in
the second item concerns positive descendant levels, where degrees in
the forward subtree equal degrees in $T$.
\end{lemma}

\begin{proof}
We count sectors between consecutive incident half-edges, with
multiplicity. For a simple cycle $C$ of length $n$ and each $x\in V(C)$,
let $s_C(x)$ be the number of sectors at $x$ lying inside $C$.
Then
\begin{equation}\label{rev:tree-disk-count}
\sum_{x\in V(C)}s_C(x)\le3n-6.
\end{equation}
Indeed, let $N$ be the number of vertices strictly inside $C$, and
let $E_D,F_D$ count the edges and bounded faces in the filled disk.
This is a finite connected plane graph, by properness and connectedness.
Euler's formula and simplicity give
\[
(N+n)-E_D+F_D=1,
\qquad
3F_D\le 2E_D-n.
\]
Combining these inequalities yields
\[
E_D\le 3N+2n-3.
\]
Every interior vertex retains its full degree, at least seven. At
a boundary vertex its degree in the disk is $s_C(x)+1$. Consequently
\[
 7N+\sum_{x\in V(C)}(s_C(x)+1)
       \le2E_D\le6N+4n-6,
\]
which proves~\eqref{rev:tree-disk-count}. Each $s_C(x)$ is at least one.

Develop an abstract rooted tree into $G$ using two branching types,
$A$ and $B$. Choose two consecutive edges at $v$, in counterclockwise
order, and label their children $0,1$, of types $A,B$, respectively.
Assign type $A$ to the root. At every nonroot
vertex $x$ of degree $d$, label the half-edge toward its parent $0$
and the other half-edges $1,\ldots,d-1$ counterclockwise. Use the rules
\begin{equation}\label{rev:tree-child-rules}
\begin{array}{c|c|c|c}
\text{type of }x&\text{child-edge indices}&\text{child labels}
                                      &\text{child types}\\ \hline
A&3,4&0,1&A,B\\
B&d-5,d-4,d-3&0,\tfrac12,1&A,A,B
\end{array}
\end{equation}

Label the root by the empty word $\varnothing$.
Recursively, if a vertex is labelled by $r$, label its child
with child label $a$ by $ra$. For each word $r$ obtained
in this way, let $v_r$ denote the corresponding vertex
in $G$, with $v_\varnothing=v$.

A nonempty word ends
in type $B$ exactly when its last label is $1$; words ending in $0$
or $\tfrac12$ have type $A$. At this stage these rules specify a map
from the abstract tree into $G$; injectivity will be proved below.

For a parent-to-child turn at $x$, let $k(x)$ be the counterclockwise
sector count from the parent half-edge to the selected child edge.
Its clockwise count is $d-k(x)\ge3$. The count $k(x)$ is also at
least three, except when $x$ has type $B$, degree seven, and selects
child $0$: then $k(x)=2$. Since a type-$B$ vertex is reached by a child-$1$ turn,
each deficient turn, except possibly the first in a downward
path block, is preceded by a turn with counterclockwise count
$4$ at a type-$A$ vertex or $d-3\ge4$ at a type-$B$ vertex.
Pair each deficiency with this preceding surplus.  The pairs are disjoint, because a deficient turn
selects child $0$ whereas a supplying turn selects child $1$.
Thus any block of $m$ internal vertices on a downward path satisfies
\begin{equation}\label{rev:tree-turn-block}
 \sum k(x)\ge3m-1,\qquad \sum(d(x)-k(x))\ge3m.
\end{equation}
At most the first vertex of the block can have an unpaired deficiency.

Develop children one at a time in breadth-first order. Suppose that the first collision occurs when a new child
of $z$ coincides with a previously used vertex $w$. Before this step,
the developed edges form an embedded finite tree.  Hence it closes a simple cycle
$C$ of length $n\ge3$ with the existing tree path from $z$ to $w$.
Let $r$ be the least common ancestor of $z,w$ in that existing tree.
If $r\ne w$, the two exceptional corners of $C$ are $r,w$.
The remaining $n-2$ corners split into two downward ancestor-path
blocks. Their bounded-side sectors are clockwise on one block and
counterclockwise on the other, since the two downward branches are
traversed in opposite directions around $C$. By
\eqref{rev:tree-turn-block}, they contribute at least $3(n-2)-1$.
The two exceptional corners contribute at least one sector each, so
\[
                  \sum_{x\in V(C)}s_C(x)\ge3n-5>3n-6,
\]
contradicting~\eqref{rev:tree-disk-count}. This includes $r=z$, when
one block is empty. If $r=w$, the old vertex $w$ is an ancestor of
$z$. There is only one exceptional corner; the other corners form
one downward block. The lower bound is then
$3(n-1)-1+1=3n-3$, again a contradiction. There is no first collision.
Induction proves global injectivity, and the union is a tree
$T\subseteq G$. Its embedding is proper because that of $G$ is proper.

Every type-$A$ nonroot vertex has degree three and every type-$B$
vertex has degree four. Every vertex has exactly one type-$B$ child
and at most three children in total. If $L_n(\sigma)$ denotes the
descendants of $\sigma$ at distance $n$, then the number of type-$B$
vertices in $L_n(\sigma)$ equals $|L_{n-1}(\sigma)|$, whereas
\[
                  |L_n(\sigma)|\le3|L_{n-1}(\sigma)|.
\]
This proves the one-third assertion, including when $\sigma=v$.

Put $\Lambda=1+\sqrt2$, $h(A)=1$, $h(B)=\sqrt2$, and assign the root
weight one. The sum of the child weights at any vertex is $\Lambda$
times its own weight. Orient the edges of $T_\sigma$ away from $\sigma$.
For each nonroot vertex $x$, assign to the edge from
its parent to $x$ the flow
\[
\theta_\sigma(\operatorname{parent}(x),x)
=\frac{h(x)}{h(\sigma)}
\Lambda^{-d_T(\sigma,x)}.
\]
Then $\theta_\sigma$ is a unit flow. The weights are bounded above and below
by positive constants, so the flow characterization of branching
number gives $\mathrm{br}(T_\sigma)\ge\Lambda$. Conversely,
the weighted sum on its level $n$ is $h(\sigma)\Lambda^n$;
level cutsets give the opposite inequality. Thus every forward subtree
has branching number $1+\sqrt2$. The standard bond-percolation identity
\cite[Theorem~6.2]{ry90}, equivalently \cite[Theorem~5.15]{LP16},
also gives the site threshold: conditional on the root being open,
associate each nonroot vertex with its parent edge. The independent
site states then give exactly the bond-percolation root cluster, so
the site survival probability is the bond survival probability times
the root-retention probability. Thus
\begin{equation}\label{rev:tree-critical}
 \mathrm{br}(T_\sigma)=1+\sqrt2,
 \qquad p_c^{\mathrm{site}}(T_\sigma)
                  =\sqrt2-1<\frac37.
\end{equation}
\end{proof}

\begin{proposition}\label{l510}
Let $G$ satisfy the assumptions of Lemma~\ref{l59}, and put
$p_T=\sqrt2-1$. For every fixed
\[
                            p_T<p<1-p_T,
\]
Bernoulli site percolation on $G$ almost surely has infinitely many
infinite $1$-clusters and infinitely many infinite $0$-clusters.
\end{proposition}

\begin{proof}
Fix such a $p$ and write $q=1-p$. Use the tree from Lemma~\ref{l59}
and its word labels.
For $D\in\{A,B\}$, let $\beta_D(t)$ be the probability that the root
of a type-$D$ forward subtree belongs to an infinite retained cluster at
retention $t$. This event includes retention of the root. By
\eqref{rev:tree-critical}, $\beta_D(p),\beta_D(q)>0$ for both types.

For a word $r$ of type $A$, 
let $E_r$ be the event that $v_r,v_{r0},v_{r1}$ are closed, that
\[
 v_{r00}\xleftrightarrow{\,0\,}\infty\text{ in }T_{v_{r00}},
 \qquad
 v_{r10}\xleftrightarrow{\,0\,}\infty\text{ in }T_{v_{r10}},
\]
and that
\[
 v_{r01}\xleftrightarrow{\,1\,}\infty\text{ in }T_{v_{r01}},
 \qquad
 v_{r11}\xleftrightarrow{\,1\,}\infty\text{ in }T_{v_{r11}}.
\]
The four forward subtrees are disjoint, their
roots have types $A,A,B,B$, and they avoid the three specified stem
vertices. Consequently
\begin{equation}\label{rev:tree-separation-probability}
       \PP_p(E_r)=q^3\beta_A(q)^2\beta_B(p)^2=:a(p)>0.
\end{equation}

On $E_r$, choose closed rays in $T_{v_{r00}}$ and $T_{v_{r10}}$ from their roots.
Together with the path
\[
                 v_{r00}-v_{r0}-v_r-v_{r1}-v_{r10}
\]
they form a closed proper simple double ray $\Gamma_r$. Its embedded
trace is a proper Jordan line and separates the plane into two components.
Orient it from the $r00$ end toward the $r10$ end. At $v_{r0}$ it
travels from child $0$ toward the parent, whereas at $v_{r1}$ it
travels from the parent toward child $0$. The counterclockwise child orders $0,1$ and
$0,\tfrac12,1$ in \eqref{rev:tree-child-rules}
show that $T_{v_{r01}}$ branches off $\Gamma_r$
on the opposite side from $T_{v_{r1\tfrac12}}$
and $T_{v_{r11}}$.
Each of these forward subtrees is connected and
disjoint from $\Gamma_r$, so it lies entirely
on the side determined by this local check. Thus $\Gamma_r$
separates $T_{v_{r01}}$ from both $T_{v_{r1\tfrac12}}$ and $T_{v_{r11}}$.
Any graph path between these sides must meet the double ray at a
vertex, by planarity. Its closed vertices prevent an open connection.

Define $r_1=\varnothing$ and
$r_{m+1}=r_m1\tfrac12$ for $m\ge1$. Every nonempty $r_m$ has type $A$.
The event $E_{r_m}$ uses only the vertex set
\[
\begin{split}
 D_m={}&\{v_{r_m},v_{r_m0},v_{r_m1}\}
       \cup V(T_{v_{r_m00}})\cup V(T_{v_{r_m01}})\\
       &\qquad\cup V(T_{v_{r_m10}})\cup V(T_{v_{r_m11}}).
\end{split}
\]
This lies in $V(T_{v_{r_m}})\setminus V(T_{v_{r_{m+1}}})$, whereas every later
$D_n$ lies inside $V(T_{v_{r_{m+1}}})$. Hence the sets $D_m$ are pairwise
disjoint, and the events $E_{r_m}$ are independent. By
\eqref{rev:tree-separation-probability} and Borel--Cantelli, infinitely
many of them occur almost surely.

For each successful $m$, take the infinite open cluster in $G$ containing
the open ray in $T_{v_{r_m01}}$. Its closed barrier $\Gamma_{r_m}$ separates
it from the whole next-spine subtree $T_{v_{r_{m+1}}}$, which contains
every later selected open ray. The selected infinite clusters are
therefore pairwise distinct. This proves the open-cluster assertion.
Apply the same argument to the colour-switched configuration, whose
parameter is $q\in(p_T,1-p_T)$, to obtain the closed-cluster assertion.
\end{proof}

\begin{remark}\label{rev:tree-quadrilateral-extension}
The same conclusions hold if minimum vertex degree seven is replaced
by minimum vertex degree five together with minimum finite-face degree
four. Here the disk count is $\sum_{x\in V(C)}s_C(x)\le2|C|-4$.
Indeed $4F_D\le2E_D-|C|$ gives
$E_D\le2N+\tfrac32|C|-2$, while interior degrees are at least five.
Use child-edge indices $2,3$ for type $A$, and $d-4,d-3,d-2$ for
type $B$, with the same child labels and types. Every clockwise turn
has at least two sectors; the sole possible counterclockwise deficiency
is a type-$B$, degree-five first-child turn with one sector. It is
paired with the preceding last-child turn, which has at least three.
The two-corner collision bound becomes $2(n-2)-1+2=2n-3>2n-4$;
the ancestor-collision bound is $2(n-1)-1+1=2n-2>2n-4$.
Thus the same injectivity argument and the same type tree apply.
Its branching number and critical probability are unchanged, as is
the ordered-branch separation proof of Proposition~\ref{l510}.
\end{remark}
\end{revision}

\begin{revision}
\section{Finite-set estimates and a nonuniqueness criterion}
\label{sect:ccp}

We use the quantitative site-percolation estimate below to prove
Lemma~\ref{l74}, then establish Proposition~\ref{la68} and deduce
the nonuniqueness criterion in Corollary~\ref{l83}.

Let $G=(V,E)$ be an infinite connected locally finite graph.
By \cite[Lemma~2.5]{perc24}, for every $\pcs(G)<p<1$ and
$\varepsilon>0$, there exists $w\in V$ such that
\[
  \PP_p(w\leftrightarrow\infty)
  \ge 1-\left(\frac{1-p}{1-\pcs(G)}\right)^{1-\varepsilon}.
\]

\begin{lemma}\label{l74}
Let $\pcs(G)<p<1$. For every $\varepsilon>0$, infinitely many vertices
$v\in V$ satisfy
\begin{equation}
 \PP_p(v\leftrightarrow\infty)
 \ge 1-\left(\frac{1-p}{1-\pcs(G)}\right)^{1-\varepsilon}.
 \label{rev:many-good}
\end{equation}
In particular, there are an infinite set $W\subseteq V$ and $c>0$ with
$\PP_p(v\leftrightarrow\infty)\ge c$ for every $v\in W$.
\end{lemma}
\begin{proof}
It is enough to consider $0<\varepsilon<1$, since for
$\varepsilon\ge1$ the displayed lower bound is nonpositive. Fix such
an $\varepsilon$ and call the positive right-hand side of
\eqref{rev:many-good} $c$.

For every nonempty finite $F\subset V$, the graph $G\setminus F$ has
only finitely many connected components: each component is joined to
$F$ by an edge, and $F$ has finite total degree. For any configuration,
an infinite open cluster exists in $G$ if and only if one exists in
$G\setminus F$. Indeed, deleting finitely many vertices from an
infinite connected locally finite cluster leaves an infinite component.
It follows that
\begin{equation}
 \pcs(G)=\min\{\pcs(Y):Y\text{ is an infinite component of }
                              G\setminus F\}.
 \label{rev:finite-deletion-critical}
\end{equation}
To see the equality directly, existence of an infinite open cluster in
$G\setminus F$ is the union of the corresponding events in its
finitely many components; below the minimum none percolates, while
above it one does. Choose an infinite component $Y$ attaining this
minimum. By \cite[Lemma~2.5]{perc24}, applied to $Y$, there is
$v\in V(Y)$ with $\PP_p^Y(v\leftrightarrow\infty)\ge c$.
The same lower bound holds
in $G$. Thus a suitable vertex exists outside every finite $F$, proving
that there are infinitely many such vertices. 
\end{proof}

\begin{proof}[Proof of Proposition~\ref{la68}]
We treat $0<p<1$; the endpoints are immediate. Put
$\theta_p(v)=\PP_p(v\leftrightarrow\infty)$ and
$E_v=\{v\leftrightarrow\infty\}\cap\mathcal A_1$.
Let $C(u)$ be the open cluster of $u$, with $C(u)=\varnothing$ when
$u$ is closed. Condition on $C(u)=S$ for a finite nonempty $S$ having
positive probability. This fixes all vertices of $S$ open and all
vertices of $\partial_G S$ closed, and leaves every
other vertex independent Bernoulli($p$). Resampling $S$ and its
external boundary independently with parameter $p$ gives an
unconditional configuration.

If $E_v$ holds in the conditioned configuration, its unique infinite
cluster avoids the resampled set and survives as a connected infinite
subgraph. Opening finitely many boundary vertices can join only
finitely many old clusters, by local finiteness; it cannot create a
second infinite cluster from finite clusters. Changes inside the
finite set $S$ have the same property. Hence $E_v$ still holds after
every such resampling. When $S=\varnothing$, resample only the closed
vertex $u$ and use the same argument. Therefore
\[
 \PP_p(E_v\mid C(u)=S)\le\PP_p(E_v)
 \qquad\text{for every finite }S
 \text{ with }\PP_p(C(u)=S)>0.
\]
Summing over these finite clusters yields
\[
 \PP_p(u\leftrightarrow\infty,E_v)
 \ge\theta_p(u)\PP_p(E_v).
\]
The same resampling argument with $\mathcal A_1$ alone, conditioning
on finite $C(v)$, gives
\[
 \PP_p(E_v)\ge\theta_p(v)\PP_p(\mathcal A_1).
\]
On $\mathcal A_1$, any two vertices in infinite open clusters are
connected. Combining the last two inequalities proves
\begin{equation}
 \PP_p(u\leftrightarrow v)
 \ge\theta_p(u)\theta_p(v)\PP_p(\mathcal A_1).
 \label{rev:uniqueness-bound}
\end{equation}
Only the stated finite-cluster resampling implication has been used;
the event $\mathcal A_1$ need not be increasing under arbitrary
changes of the configuration.
\end{proof}

\begin{lemma}\label{rev:finite-to-unique}
For Bernoulli site percolation on a connected locally finite graph at
$0<p<1$, if the number of infinite open clusters is finite and positive
with positive probability, then
$\PP_p(\mathcal A_1)>0$.
\end{lemma}
\begin{proof}
On the event that there are finitely many but at least one
infinite open cluster, some finite ball meets all of them. By countability, there is a deterministic finite connected
ball $F$ for which this occurs with positive probability. Force every
vertex of $F$ open. All old infinite clusters are then joined through
$F$. Only finitely many old clusters are adjacent to $F$, by local
finiteness, so this modification cannot create a further infinite
cluster from finite clusters. The modified configuration has exactly
one infinite open cluster. The finite-energy property, here simply
the positive probability $p^{|F|}$ of setting the finite ball open,
implies $\PP_p(\mathcal A_1)>0$.
\end{proof}

\begin{proof}[Proof of Corollary~\ref{l83}]
Fix $\pcs(G)<p<p_{\mathrm{conn}}(G)$. The set of parameters in
definition~\eqref{df87} is downward closed by monotonicity of
connection probabilities. The definition of its supremum therefore
implies
\begin{equation}
 \lim_{n\to\infty}\sup_{d_G(x,y)\ge n}
                   \PP_p(x\leftrightarrow y)=0.
 \label{rev:uniform-decay-p}
\end{equation}
Lemma~\ref{l74} supplies an infinite set $W$ and $c>0$ such that
$\theta_p(v)\ge c$ on $W$. Fix $u\in W$ and choose $v_n\in W$ such that
$d_G(u,v_n)\to\infty$.
Proposition~\ref{la68} gives
\[
 \PP_p(u\leftrightarrow v_n)
 \ge c^2\PP_p(\mathcal A_1).
\]
The left-hand side tends to zero by~\eqref{rev:uniform-decay-p}, so
$\PP_p(\mathcal A_1)=0$. Lemma~\ref{rev:finite-to-unique} excludes
a finite positive number of infinite clusters with positive
probability. Finally, $p>\pcs(G)$ implies that an infinite cluster
exists almost surely: existence has positive probability and is a
tail event, since finitely many changes cannot remove or create all
infinite clusters in a locally finite graph. The number of infinite
clusters is therefore almost surely infinite.
\end{proof}
\end{revision}

\begin{revision}
\section{Thin ends and the low-density interval}\label{sect:planarg}

We first prove the general thin-end nonuniqueness theorem,
Theorem~\ref{rev:thin-nonunique}. We then use induced-cycle
estimates and independent thinning to obtain the low-density
nonuniqueness result in Proposition~\ref{rev:lower-interval}.

\begin{mainrevisionblock}
The vertex-degree of an end is the supremum of the numbers
of pairwise vertex-disjoint rays belonging to it.
An end is called thin if its vertex-degree is finite;
see~\cite[Section~2.1]{BCP}.
\end{mainrevisionblock}

\begin{proof}[\mainrevision{Proof of Theorem~\ref{rev:thin-nonunique}}]
By Lemma~\ref{l74}, there exist an infinite set
$W\subseteq V(X)$ and a constant $c>0$ such that
$\theta_s(w)\ge c$ for every $w\in W$. There is an end $\omega$ such that, for every finite
$K\subseteq V(X)$, the component $\omega(K)$ contains
infinitely many vertices of $W$. For an elementary construction,
take a rooted locally finite spanning tree. Repeatedly choose a child
whose descendant subtree contains infinitely many $W$-vertices. This
gives a ray $R$. For any finite $F$, a sufficiently far descendant subtree
along $R$ is disjoint from $F$, connected, and still contains infinitely
many $W$-vertices. It lies in $\omega(F)$.

Let $k<\infty$ be the vertex-degree of $\omega$.
By \cite[Definition~2 and Lemma~3]{BCP}, there
exists a sequence of separations $(A_i,B_i)$
capturing $\omega$. Set
\[
S_i=V(A_i)\cap V(B_i),
\qquad D_i=\omega(S_i).
\]
The sets $S_i$ are pairwise disjoint, with
$|S_i|=k$, and $D_{i+1}\subseteq D_i$.
Moreover, for every finite $F\subseteq V(X)$,
we have $D_i\cap F=\varnothing$ for all
sufficiently large $i$.

Fix $u\in W$. By discarding finitely many initial
terms of the sequences $(S_i)$ and $(D_i)$ and
reindexing, we may assume that
$u\notin D_1\cup S_1$.
For each $n$ choose $v_n\in W\cap D_n$. Every $u$--$v_n$ path meets
each $S_i$, $1\le i\le n$. Independence on the disjoint separators yields
\begin{equation}
 \PP_s(u\leftrightarrow v_n)
 \le\prod_{i=1}^n[1-(1-s)^{|S_i|}]
 \le[1-(1-s)^k]^n\longrightarrow0.
\label{rev:thin-decay}
\end{equation}
The lower bound~\eqref{rev:uniqueness-bound} is at least
$c^2\PP_s(\mathcal A_1)$. Hence $\PP_s(\mathcal A_1)=0$.

If the number of infinite clusters were finite and positive with positive
probability, some deterministic finite connected ball would meet all of
them with positive probability, by countability. Forcing that ball open
merges them into one, and cannot create another infinite cluster from
finite clusters by local finiteness. Finite energy would imply
$\PP_s(\mathcal A_1)>0$, a contradiction. Finally $s>\pcs(X)$ implies
existence of an infinite cluster almost surely. Thus their number is
almost surely infinite. No assertion at $s=\pcs(X)$ is needed.
\end{proof}

\begin{lemma}\label{rev:four-rays}
Let $X$ be a connected locally finite graph properly
embedded in the plane. If $X$ contains four pairwise
vertex-disjoint rays, then $X$ contains an induced
double ray.
\end{lemma}

\begin{proof}
Take a closed round disk $D$ containing the initial
vertices of the four rays in its interior.
By properness, each ray has a last intersection
point with $\partial D$. Denote these distinct
points by $x_1,x_2,x_3,x_4$ in cyclic order,
and let $R_i$ be the exterior tail starting
at $x_i$.

Join $x_2$ and $x_4$ by a simple arc whose
interior lies in the interior of $D$, and let
$J$ be its union with $R_2$ and $R_4$.
In the one-point compactification
$\mathbb R^2\cup\{\infty\}$, the set
$J\cup\{\infty\}$ is a simple closed curve.
By the Jordan curve theorem, $J$ separates
$\mathbb R^2$ into two components, with
$R_1$ and $R_3$ on opposite sides.

Every edge joining $R_1$ to $R_3$ must
intersect $J$. By planarity, such an edge
cannot intersect $R_2$ or $R_4$, and hence
must intersect the joining arc inside $D$.
Properness implies that only finitely many
edges of $X$ intersect $D$. We may therefore
choose graph rays $R'_1$ and $R'_3$ by
deleting sufficiently long initial segments
of $R_1$ and $R_3$, so that no edge of $X$
joins $R'_1$ to $R'_3$.

Let $H_1=X[V(R'_1)]$ and
$H_3=X[V(R'_3)]$. Both graphs are infinite,
connected and locally finite. Since $H_1$ and $H_3$ are infinite, connected and
locally finite, K\"onig's infinity lemma applied
to shortest paths from a fixed vertex yields a
geodesic ray in each graph. Choose such rays $P$ in $H_1$
and $Q$ in $H_3$. Since $H_1$ and $H_3$
are induced subgraphs of $X$, both $P$ and
$Q$ are induced rays in $X$. Moreover,
no edge of $X$ joins $P$ to $Q$.

Choose a shortest path
$z_0,z_1,\ldots,z_d$ in $X$ joining $P$
to $Q$, where $z_0\in V(P)$,
$z_d\in V(Q)$, and $d\ge2$.
By shortestness, this path is induced,
no $z_i$ with $2\le i\le d-1$ is
adjacent to $P$, and no $z_i$ with
$1\le i\le d-2$ is adjacent to $Q$.

Since $X$ is locally finite, $z_1$ has
only finitely many neighbours on $P$.
Let $p_*$ be its last neighbour along
$P$, and let $P_*$ be the infinite tail
of $P$ starting at $p_*$. Similarly,
let $q_*$ be the last neighbour of
$z_{d-1}$ along $Q$, and let $Q_*$
be the infinite tail of $Q$ starting
at $q_*$. The path
\[
p_*,z_1,\ldots,z_{d-1},q_*
\]
is induced, and its internal vertices
have no additional neighbours on
$P_*$ or $Q_*$. Since the two tails
are induced and no edge joins them,
their union with this path is an
induced double ray. The argument
also applies when $d=2$.
\end{proof}

\begin{lemma}\label{rev:double-cycles}
Let $X$ be a connected locally finite graph containing an induced double
ray whose two tails belong to the same end of $X$. Each fixed edge of that
double ray belongs to induced cycles of arbitrarily large length.
\end{lemma}
\begin{proof}
Write the double ray as $(v_i)_{i\in\ZZ}$ and fix $e=v_0v_1$. For $n\ge2$
let $P_n=v_{-n},\ldots,v_{n+1}$ and remove the finite set
\[
       F_n=V(P_n)\cup \partial_X\{v_{-n+1},\ldots,v_n\}.
\]
The two tails beyond the endpoints avoid $F_n$ because the double ray is
induced. They belong to the same infinite component of $X\setminus F_n$
because the two tails belong to the same end. Hence the endpoints of $P_n$
can be joined by a path $Q_n$ whose internal
vertices avoid $F_n$. Choose a shortest such path $Q_n$. It has no chord,
and its interior has no edge to the interior of $P_n$. Therefore
$P_n\cup Q_n$ is an induced cycle containing $e$, with length tending
to infinity.
\end{proof}

\begin{lemma}\label{rev:bounded-end-degree}
Let $X$ be connected, locally finite and properly embedded. If no fixed
edge belongs to induced cycles of arbitrarily large length, every end
of $X$ has vertex degree at most three.
\end{lemma}
\begin{proof}
Suppose four pairwise vertex-disjoint rays belong to
the same end $\omega$. The construction in
Lemma~\ref{rev:four-rays} yields an induced double ray
whose two tails still belong to $\omega$.
By Lemma~\ref{rev:double-cycles}, a fixed edge of this
double ray belongs to induced cycles of arbitrarily
large length, contrary to the hypothesis.
\end{proof}

\begin{lemma}\label{rev:hereditary-end-degree}
If every end of a locally finite graph $X$ has vertex degree at most $k$,
the same holds for every connected subgraph $Y\subseteq X$.
\end{lemma}
\begin{proof}
If $k+1$ disjoint rays belong to one end of $Y$, then for every finite
$F\subset V(X)$ their tails are connected in $Y\setminus(F\cap V(Y))$,
hence in $X\setminus F$. They belong to one end of $X$, contradicting
its degree bound. The map from ends of $Y$ to ends of $X$ need not be
injective; the argument requires no such assertion.
\end{proof}

\begin{proposition}\label{rev:lower-interval}
For $G$ as in Theorem~\ref{mt1} and every fixed
\[
                   \pcs(G)<p<q_0=\frac2{3\sqrt3},
\]
there are almost surely infinitely many infinite $p$-open clusters.
\end{proposition}
\begin{proof}
Let $H$ be the finite-face triangulation from
Lemma~\ref{rev:augmentation}. Choose deterministic parameters
\begin{equation}
                     \pcs(G)<p_0<p<p'<q_0.
\label{rev:three-levels}
\end{equation}
Since $p'<q_0$, Lemma~\ref{rev:polygon-count} applied to $H$ gives,
for each edge $e$ of $H$, probability zero that $e$ lies on $p'$-open
induced cycles of unbounded length. As $E(H)$ is countable, almost
surely no edge of $H$ lies on such cycles of unbounded length. Each component of $H_{p'}$ is an induced subgraph of $H$, so
Lemma~\ref{rev:bounded-end-degree} gives end degree at most three in
every such component. Each component of $G_{p'}$ is a subgraph of one of them,
and inherits the same bound by Lemma~\ref{rev:hereditary-end-degree}.

Condition on the entire $p'$-open vertex set. The components of $G_{p'}$
now form a fixed countable family of graphs. Retention with probability $p_0/p'$
gives the conditional law of $G_{p_0}$. Since $p_0>\pcs(G)$, global
percolation at this lower level has probability one, and consequently
conditional probability one for almost every $p'$-configuration.
For each such configuration, at least one component $Y$ has positive
conditional probability of percolation at retention $p_0/p'$.
Otherwise countably many null events would imply no global percolation.
\thisrevision{On the fixed locally finite graph $Y$, existence of an infinite retained
cluster is a tail event for the independent retention variables. Its positive
conditional probability is therefore one.} Therefore
\begin{equation}
                         \pcs(Y)\le p_0/p'<p/p'<1.
\label{rev:host-margin}
\end{equation}
This component can be chosen measurably as the component of the first
vertex, in a fixed enumeration, whose conditional probability of
belonging to an infinite $p_0/p'$-retained cluster is positive.

Conditional retention at $s=p/p'$ is iid on $Y$. It is strictly
supercritical by~\eqref{rev:host-margin}, so
Theorem~\ref{rev:thin-nonunique} gives infinitely many infinite retained
clusters inside $Y$ almost surely. They cannot merge through another
component of $G_{p'}$: every $p$-open path is also $p'$-open. They are thus
distinct components of $G_p$. Integrating the conditional probability
one assertion proves the proposition. The auxiliary parameter $p_0<p$ ensures that
$\pcs(Y)\le p_0/p'<p/p'$, so that Theorem~\ref{rev:thin-nonunique}
applies at the strictly supercritical retention parameter $p/p'$.
\end{proof}
\end{revision}

\begin{revision}
\section{Infinite ends and planar separation}\label{sect:5}

We develop finite-cut and planar-separation arguments for the
high-density regime. The main conclusion,
Proposition~\ref{rev:high-density}, deduces nonuniqueness from
an infinite closed cluster with infinitely many ends and a
finite-deletion percolation condition.

\subsection{Finite vertex cuts}

\begin{lemma}\label{rev:cutsets}
Let $X$ be an infinite connected locally finite graph.
Suppose that there exists a nonnegative integer $k$ such that
$X$ contains no $k+1$ pairwise vertex-disjoint rays.
Then $\pcs(X)=1$.
\end{lemma}
\begin{proof}
We first show that, for every finite $F\subset V(X)$, there is
a set $S\subset V(X)\setminus F$ with $|S|\le k$ such that
every component of $X\setminus S$ meeting $F$ is finite.

Delete $F$ and add a vertex $r$ adjacent to every vertex of
$\partial_X F$. If no such $S$ existed, no set of at most $k$
vertices other than $r$ could separate $r$ from infinity.
Menger's theorem would then give $k+1$ paths from $r$ to the
boundary of each finite ball, meeting only at $r$. By local
finiteness and K\"onig's lemma, these paths yield $k+1$ rays
meeting only at $r$. Deleting $r$ gives $k+1$ pairwise
vertex-disjoint rays in $X$, a contradiction.

Enumerate $V(X)=\{v_1,v_2,\ldots\}$. Apply the preceding assertion
recursively with
\[
F_n=\{v_1,\ldots,v_n\}\cup\bigcup_{j<n}S_j
\]
to obtain $S_n\subset V(X)\setminus F_n$, $|S_n|\le k$.
Thus the sets $S_n$ are pairwise disjoint, and every component
of $X\setminus S_n$ meeting $F_n$ is finite.

Fix $0\le s<1$. The events that all vertices of $S_n$ are closed
are independent, and each has probability
\[
(1-s)^{|S_n|}\ge (1-s)^k>0.
\]
By the Borel--Cantelli lemma, infinitely many of these events
occur almost surely. For each $i$, choose such an $n\ge i$.
Since $v_i\in F_n$ and all vertices of $S_n$ are closed,
the open cluster of $v_i$, if nonempty, is contained in a finite
component of $X\setminus S_n$. Hence almost surely every open
cluster is finite. Since this holds for every $s<1$,
we have $\pcs(X)=1$.
\end{proof}

\begin{proposition}\label{rev:finite-ended-pc}
If $X$ is infinite, connected, locally finite, properly embedded and has finitely
many ends, and no fixed edge lies on induced cycles of unbounded length,
then $\pcs(X)=1$.
\end{proposition}
\begin{proof}
If $X$ is one-ended, Lemmas~\ref{rev:four-rays} and
\ref{rev:double-cycles} exclude four disjoint rays. Apply
Lemma~\ref{rev:cutsets} with $k=3$. In the finitely many ends case, delete
a finite set separating the ends into one-ended infinite components.
There are only finitely many components by local finiteness. The induced
cycle condition is inherited by induced subgraphs, so each infinite
component has threshold one. Finite deletion cannot change the existence
of an infinite retained cluster. Hence $X$ also has threshold one.
\end{proof}

\begin{theorem}\label{rev:infinite-ends}
Let $X$ be connected, locally finite and properly embedded in the plane.
Fix $q>\pcs(X)$. If for every edge $e$
\begin{equation}
 \lim_{m\to\infty}\PP_q(e\text{ belongs to an induced open cycle
 of length at least }m)=0,
\label{rev:no-long}
\end{equation}
then almost surely some infinite $q$-open cluster has infinitely many ends.
\end{theorem}
\begin{proof}
Countability of the edge set and~\eqref{rev:no-long} imply that almost surely
no fixed edge belongs to open induced cycles of unbounded length.
Every cluster, and every component obtained by deleting vertices from it,
is an induced subgraph of $X$. On the event $B$ that all infinite clusters
have finitely many ends, Proposition~\ref{rev:finite-ended-pc} therefore
gives internal site threshold one for every infinite cluster.
Choose $\pcs(X)<q'<q$. Independently retain each $q$-open vertex with
probability $q'/q$. Conditional on a configuration in $B$, every one of
its countably many clusters becomes a union of finite components almost
surely. The thinned field has Bernoulli($q'$) law and, since
$q'>\pcs(X)$, has an infinite cluster almost surely.
If $\PP_q(B)>0$, this contradicts the fact that, conditional
on $B$, the thinned field has no infinite cluster almost surely.
The contradiction implies $\PP_q(B)=0$.
\end{proof}

Theorem~\ref{rev:infinite-ends} will be applied to a plane graph whose
finite faces are triangular. Its hypotheses concern induced cycles in
that plane graph itself.

\subsection{Separation from finitely many open clusters}\label{rev:separation}

In this section $X$ is connected and properly embedded, and every finite
face is triangular. Infinite faces are allowed.

\begin{lemma}\label{rev:quotient}
Fix a deterministic two-colouring of $V(X)$.
A monochromatic cluster is a connected component of the subgraph
induced by either colour class. Let $Q$ be the simple graph
whose vertices are these clusters, with two distinct clusters
adjacent if and only if an edge of $X$ joins them.
Then $Q$ is a tree.
\end{lemma}
\begin{proof}
The graph $Q$ is connected and bipartite. A finite simple cycle of $Q$ would lift
to a finite simple cycle $J$ of $X$: choose the intercluster edges and
join consecutive attachment vertices by simple paths within each of the
distinct clusters. The bounded disk of $J$ contains finitely many
triangular faces. Map each vertex to its monochromatic cluster. Each
triangle maps either to a point or to a single $Q$-edge traversed twice,
since equal-coloured adjacent vertices lie in the same cluster. Sum the
face boundaries modulo two in the vector space generated by $Q$-edges.
All triangle images are zero, whereas the image of $J$ contains every
edge of the proposed $Q$-cycle once. This is a contradiction.
\end{proof}

\begin{lemma}\label{rev:neighbours}
If $A$ is an open cluster and $\xi$ a closed cluster, then
$\partial_X A\cap V(\xi)$, when nonempty, induces a connected graph.
\end{lemma}
\begin{proof}
Recolour all vertices outside $A\cup \partial_X A$ open, keeping $\partial_X A$ closed.
Then $A$ is still a cluster, and the new closed clusters are exactly the
components of $X[\partial_X A]$. If two distinct such components contained
vertices of old $\xi$, an old $\xi$-path would project to a walk joining
them in the new quotient without visiting $A$. But both components are
adjacent to $A$, contradicting the tree property in
Lemma~\ref{rev:quotient}. Since every vertex of $\partial_X A$ was already
closed, a component meeting $\xi$ is contained in $\xi$. This proves
the assertion.
\end{proof}

\begin{lemma}\label{rev:two-ends}
Let $\xi$ be a connected locally finite graph properly embedded
in $\RR^2$. Write $\operatorname{Emb}(\xi)$ for the union of its
embedded vertices and edges, and let $U$ be a component of
$\RR^2\setminus\operatorname{Emb}(\xi)$. Say that $U$
is incident to an end $\eta$ if its closure meets the embedded component
$\eta(S)$ of $\xi\setminus S$ for every finite $S\subset V(\xi)$.
Then $U$ is incident to at most two ends of $\xi$.
\end{lemma}
\begin{proof}
Suppose $U$ is incident to three distinct ends, and choose a
finite vertex set $S$ separating them. Choose a properly embedded
tripod in $\xi$ whose arms represent these ends.
By the Jordan curve theorem, $U$ lies in a complementary region
bounded by two arms. Let $L$ be the double ray formed by these
arms, and let $\omega$ be the end represented by the third arm.

Choose a finite set $S'\supseteq S$ such that every vertex of $L$
outside $S'$ belongs to one of the two components of
$\xi\setminus S$ corresponding to the first two ends.
Then $\omega(S')\subseteq\omega(S)$ has no vertex on $L$, and its
embedded image is disjoint from $L$ by planarity.
Since it is connected and contains a tail of the third arm,
it lies entirely on the side of $L$ opposite $U$.
Thus $\overline U$ does not meet $\omega(S')$, contradicting
the assumption that $U$ is incident to $\omega$.
\end{proof}

\begin{lemma}\label{rev:common-cut}
Suppose a colouring of $X$ has only finitely many infinite open clusters
$C_1,\ldots,C_k$. If a closed cluster $\xi$ has an end $\eta$ incident to
none of the complementary components containing the $C_j$, then some
finite $K\subset V(X)$ separates a tail of $\eta$ from every $C_j$.
\end{lemma}
\begin{proof}
Nonincidence and finiteness of the list give a finite $S\subset V(\xi)$
such that $D=\eta(S)$ is adjacent to none of the infinite open clusters.
Indeed, an edge from $D$ to $C_j$ would place its endpoint in the closure
of the complementary component containing $C_j$.

In the cluster tree $Q$, let $A_j$ be the first open cluster on the unique
path from $\xi$ to $C_j$. Form $K$ from $S$, all vertices of all finite
open clusters adjacent to $S$, and all vertices of each finite $A_j$. Then $K$ is finite, by local finiteness of $X$ and the finite list of targets $C_1,\ldots,C_k$.
$K$ is disjoint from $D$.

Suppose a finite path in $X\setminus K$ joined $D$ to $C_j$.
Any excursion from $\xi$ that returns to $\xi$ must leave and
return through the same open neighbour $A$ of $\xi$ in $Q$,
since $Q$ is a tree. The first departure point lies in $D$,
since the path starts in $D$ and stays in $\xi\setminus S$
until it leaves $\xi$. Assume inductively that the current
departure point lies in $D$. Thus $A$ is finite, since $D$
has no neighbour in an infinite open cluster. As $A$ has not
been deleted, it is not adjacent to $S$.
Lemma~\ref{rev:neighbours} connects its two attachment points
within $\xi\setminus S$, so the return point is still in $D$.
The path then stays in $D$ until its next departure from $\xi$.
Finally, the last departure towards $C_j$ must use $A_j$.
If $A_j$ is finite it was deleted; if infinite it cannot
neighbour $D$. Both are impossible. Thus the ambient component
containing $D$ misses every $C_j$.
\end{proof}

\begin{proposition}\label{rev:high-density}
Suppose at parameter $p$ that almost surely some infinite closed cluster
has infinitely many ends, and that every infinite component after every
finite vertex deletion contains an infinite open cluster. Then almost
surely there are infinitely many infinite open clusters.
\end{proposition}
\begin{proof}
On a configuration with only finitely many infinite open clusters, each
is contained in a complementary component of the infinite-ended closed
cluster $\xi$. Lemma~\ref{rev:two-ends} shows that these components are
incident to only finitely many ends. Choose another end of $\xi$.
Lemma~\ref{rev:common-cut} gives a finite deletion whose component containing
that end's tail is infinite and contains no infinite open cluster, a
contradiction. 
\end{proof}
\end{revision}

\begin{revision}
\section{The triangular polygon estimate and augmentation}\label{sect:6}

This section supplies the geometric input for both parameter
ranges. We prove an induced-cycle counting estimate and construct
a triangulation of the finite faces of the original graph.

\begin{lemma}\label{rev:polygon-count}
Let $X$ be connected and properly embedded in the plane, with minimum
degree at least $d\ge7$ and triangular finite faces.
For every edge $e$, the number $c_n(e)$ of induced $n$-cycles containing $e$
satisfies, for $n\ge4$,
\begin{equation}
       c_n(e)\le2\left(\frac{(1+r)^{1+r}}{r^r}\right)^n,
       \qquad r=\frac1{d-5}.
\label{rev:cycle-count}
\end{equation}
In particular, when $d=7$, condition~\eqref{rev:no-long} holds at every
$q<q_0:=2/(3\sqrt3)$.
\end{lemma}
\begin{proof}
Let $P$ be an induced $n$-cycle, $n\ge4$, and let $m$ count the vertices
strictly inside it adjacent to its boundary. Write $N=m+k$ for the total number of vertices strictly inside
$P$, and let $E_I$ denote the number of edges in the subgraph
of $X$ induced by these vertices. By Euler's formula, the triangulated disk bounded by $P$ has
$3N+2n-3$ edges. Its interior graph is connected: the interior neighbours
around each boundary vertex form a fan path, adjacent boundary vertices'
fan paths share the third vertex of their boundary edge, and every deeper
component has a path to this first layer. Every first-layer vertex is on
the unbounded face of the interior graph, through its edge to $P$.

If $N\ge2$, its outer facial walk has length at least $m$, counting
repetitions. Euler's inequality gives $E_I\le3N-3-m$.
The number of edges between interior and boundary is
$3N+n-3-E_I$. Since every interior vertex retains its full degree,
\[
 dN\le2E_I+(3N+n-3-E_I)
       \le6N+n-6-m.
\]
Thus $n\ge(d-5)m+(d-6)k+6$. If $N=1$, the disk is a wheel and $n\ge d$,
which also implies $(d-5)m\le n$. The case $N=0$ is impossible for a
chordless polygon of length at least four in a triangular disk. We have
proved in all cases that $m\le rn$.

We now adapt the percolation counting argument from the proof of
\cite[Theorem~7]{cp19} to induced cycles in the present setting.
In auxiliary Bernoulli($t$) percolation say that $P$ occurs when its
boundary is open and its first interior layer is closed. Its occurrence
probability is $t^n(1-t)^m\ge[t(1-t)^r]^n$. At most two occurring
induced polygons contain a fixed edge, one for each interior side.
\thisrevision{Indeed, orient two distinct polygons with the same interior side so
that they traverse the fixed edge in the same direction. Follow their
common initial segment to its last vertex. At that vertex the two
bounded-side sectors start at the same incoming half-edge and are
swept in the same cyclic direction. The outgoing half-edge of the
smaller sector thus enters the interior sector of the other polygon.} Its endpoint cannot be a different boundary vertex of
that polygon, since this would be a chord. It is therefore strictly
inside and adjacent to that polygon's boundary. It must be both open
and closed, a contradiction.
Taking expectations yields $c_n(e)[t(1-t)^r]^n\le2$. Optimize at
$t=1/(1+r)$ to get~\eqref{rev:cycle-count}. For $d=7$ the base is
$3\sqrt3/2$, and the union bound
\[
 \PP_q(e\text{ belongs to an induced open cycle of length}\ge m)
 \le2\sum_{n\ge m}\left(\frac{3\sqrt3}{2}q\right)^n
\]
converges to zero for $q<q_0$.
\end{proof}

\begin{lemma}\label{rev:augmentation}
Every graph $G$ in Theorem~\ref{mt1} has a simple properly embedded
locally finite augmentation $H$ on the same vertices, with
$G\subseteq H\subseteq G_*$, minimum degree at least seven and all finite
faces triangular.
\end{lemma}
\begin{proof}
We recall the elementary curvature fact used here. At each vertex, assign angle $(m-2)\pi/m$ to each sector
belonging to a finite face of degree $m$, and angle $\pi$
to each sector belonging to an infinite face.
Count each sector separately, so that a face occupying
several sectors at the same vertex contributes once
for each of those sectors.
The curvature $2\pi$ minus the incident angle sum is nonpositive at every
vertex of degree at least seven. \mainrevision{Thus \eqref{eq:GB-ineq} applies to every simple cycle $C$.}
\begin{thisrevisionblock}
By Lemmas~\ref{lem:kappa-nonpos} and~\ref{lc29}, every finite face
$f$ has a simple boundary cycle, of length $m=|f|$. Suppose an edge
joins nonconsecutive boundary vertices outside $f$. Together with
the two boundary arcs it forms a theta graph. Since the edge lies outside $f$, the Jordan curve theorem implies
that one of the two cycles formed by the edge and a boundary arc
contains $f$ in its bounded interior. Write $n$ for the length of the enclosing cycle.
Both boundary arcs have length at least two, so $n<m$. At each of the
$n$ vertices at least one sector of $f$ lies inside this cycle.
Its boundary angle sum is at least $n(m-2)\pi/m>(n-2)\pi$, contradicting
\mainrevision{\eqref{eq:GB-ineq}}. This excludes an external chord.
\end{thisrevisionblock}

Triangulate each finite facial polygon by diagonals in its interior.
None is an already existing edge, by the external-chord assertion.
Nor can diagonals in two different faces duplicate an endpoint pair.
Otherwise add just one of these diagonals: the new graph is still simple,
proper and of minimum degree at least seven, but the other unchanged face
now has an external chord, a contradiction.

Finally, any compact disk meets only finitely many face closures. Faces
whose boundary meets it are incident to the finitely many original edges
and vertices meeting the disk. A face whose interior meets the disk but
whose boundary does not must contain the entire connected disk, and there
is at most one such face. Each finite face receives finitely many edges.
The simultaneous augmentation is therefore proper and has all the asserted properties.
\end{proof}
\end{revision}

\begin{revision}
\section{Uniform Percolation and Completion of the Proof}\label{sect:up}

\begin{mainrevisionblock}
This section proves Theorem~\ref{rev:sprinkled-stability} and uses it
to complete the proof of Theorem~\ref{mt1}. We first establish a
uniform hitting estimate from vertex expansion, which verifies the
hypothesis of Theorem~\ref{rev:sprinkled-stability} for the planar
graphs in the required parameter range. We then prove the general
stability theorem and apply it to the finite-face triangulation.
\end{mainrevisionblock}

\begin{mainrevisionblock}
Recall the definitions of $a_0$ and $q_0$ in
\eqref{eq:a0-q0}, and put
\begin{equation}\label{rev:constants}
 \varphi=\frac{1+\sqrt5}{2},\qquad
 h_0=\frac{2+\varphi}{4(1+\varphi)}.
\end{equation}
Then $a_0=1/(1+\varphi)$.
\end{mainrevisionblock}
By \cite[Theorem~2 and Corollary~17]{HP19}, every infinite connected
locally finite graph $X$ properly embedded in the plane with minimum
degree at least seven satisfies
\begin{equation}\label{rev:hp-estimates}
 |\partial_X U|\ge\varphi|U|
 \quad\text{for every finite }U\subset V(X),
 \qquad \pcs(X)\le h_0.
\end{equation}
These estimates do not require triangular finite faces. In particular,
$h_0<a_0<q_0<1/2$.

\begin{lemma}[Uniform intersection from vertex expansion]
\label{rev:uniform-hitting}
Let $X$ be a countable locally finite graph. Suppose that
\begin{equation}\label{rev:expansion-assumption}
                 |\partial_X U|\ge\kappa|U|-c
\end{equation}
for every finite $U\subset V(X)$, where $\kappa>0$ and $c\ge0$.
Fix $t\in(1/(1+\kappa),1]$, and choose $\lambda>0$ sufficiently small
that
\[
                 r=(1-t)e^\lambda+t e^{-\kappa\lambda}<1.
\]
Then every finite deterministic $A\subset V(X)$ satisfies
\begin{equation}\label{rev:uniform-hitting-bound}
 \PP_t\bigl(A\text{ meets no infinite open cluster of }X\bigr)
 \le\min\left\{1,\frac{e^{\lambda c}r^{|A|}}{1-r}\right\}.
\end{equation}
\end{lemma}

\begin{proof}
The required choice of $\lambda$ exists because the displayed expression
for $r$ equals one at $\lambda=0$ and has derivative
$1-t-\kappa t<0$ there.

Reveal the states of $A$ in a fixed order. Then explore the union of
the open clusters meeting $A$, querying an unrevealed neighbour of the
discovered open set at each step. Stop when no unrevealed neighbour of the discovered open set remains.
Resolve choices using a fixed
enumeration of $V(X)$, always taking the eligible vertex with smallest
index. Every vertex is queried at most once, and the next query is
chosen before its state is revealed. Thus the queried states form an
iid Bernoulli($t$) sequence. If exploration terminates, extend this
sequence by fresh independent Bernoulli($t$) variables.

Let $n_j$ and $m_j$ be the numbers of open and closed outcomes among the
first $j$ variables of this extended sequence, and set
$Y_j=m_j-\kappa n_j$. Independence gives
\[
 \mathbb E e^{\lambda Y_j}=r^j,
 \qquad
 \PP(Y_j\ge-c)\le e^{\lambda c}r^j.
\]
On the event in the left side of~\eqref{rev:uniform-hitting-bound},
the union $U$ of the open clusters meeting $A$ is a finite union of
finite clusters. Its boundary is finite by local finiteness, so
exploration terminates at a finite time $\tau\ge|A|$. At termination,
$n_\tau=|U|$, and every vertex of $\partial_X U$ has been queried and
found closed. Hence $Y_\tau\ge-c$ by
\eqref{rev:expansion-assumption}. Therefore
\[
 \PP_t\bigl(A\text{ meets no infinite open cluster}\bigr)
 \le\sum_{j\ge|A|}\PP(Y_j\ge-c)
 \le\frac{e^{\lambda c}r^{|A|}}{1-r}.
\]
Combining this with the trivial bound one proves the lemma. The argument
also includes $U=\varnothing$.
\end{proof}

\begin{corollary}\label{rev:finite-deletion}
Let $G$ satisfy the hypotheses of Theorem~\ref{mt1}. Every infinite
component $J$ of $G\setminus F$, for a finite $F\subset V(G)$, satisfies
\[
                            \pcs(J)\le a_0.
\]
For every fixed $p\in(a_0,1]$, almost surely every such component,
simultaneously over all finite $F$, contains an infinite $p$-open cluster.
\end{corollary}

\begin{proof}
Every neighbour of a vertex of $J$ outside $J$ lies in $F$. Consequently,
for every finite $U\subset V(J)$, \eqref{rev:hp-estimates} gives
\[
                |\partial_J U|\ge\varphi|U|-|F|.
\]
Apply Lemma~\ref{rev:uniform-hitting} to deterministic finite subsets
$A\subset V(J)$ whose cardinalities tend to infinity. If $J$ has no
infinite open cluster, each such $A$ misses all infinite open clusters.
The right side of~\eqref{rev:uniform-hitting-bound} tends to zero for
every $p>a_0$, so $J$ percolates almost surely at every such fixed
parameter. This also implies $\pcs(J)\le a_0$. There are only countably
many finite subsets of the countable vertex set, and only countably
many components for each deletion. Taking their countable intersection
proves the simultaneous assertion.
\end{proof}

\begin{mainrevisionblock}
We now prove Theorem~\ref{rev:sprinkled-stability}. The argument
adapts the duplication, finite-modification, and exploration method of
Schonmann~\cite{Schonmann99} to the graph pair $G\subseteq H$.
\end{mainrevisionblock}

\begin{proof}[\mainrevision{Proof of Theorem~\ref{rev:sprinkled-stability}}]
If $b=1$, then $H_1=H$ is connected and infinite.
Applying~\eqref{rev:uniform-hit-assumption} to deterministic sets
of arbitrarily large size shows that $\PP(I_a=\varnothing)=0$.
This proves the assertion in that case. Henceforth assume $b<1$.

Call an infinite component $C$ of $H_b$ bad if it contains no infinite
component of $G_a$, and define
\begin{equation}\label{rev:boundary-candidates}
 W(C)=\{z\in V\setminus C:z\sim_H C\text{ and }z\sim_G I_a\}.
\end{equation}
Every vertex outside a component of $H_b$ and adjacent to it is $b$-closed.
We rule out bad components according as $W(C)$ is empty, finite, or
infinite.

Fix an enumeration of $V$. In every exploration, process each
vertex at most once, always choosing the eligible vertex of
smallest index. This rule is fair: every vertex of the explored
component and every neighbour of that component is eventually
processed.

\medskip
\noindent\emph{Step 1: no bad component has $W(C)=\varnothing$.}
Fix $x\in V$ and an integer $n\ge1$. With a uniform field $U'$,
explore the component of $H_b$ containing $x$ one vertex at a time. Query $U'(x)$
first, stopping immediately if $x$ is closed. Otherwise maintain the
finite discovered open component of $x$ and query an unrevealed
$H$-neighbour of it at each step. Stop when the exploration is exhausted
or when exactly $n$ open vertices have been discovered. A finite component
and its boundary are finite, whereas fairness ensures that an infinite
component eventually supplies $n$ discovered open vertices. Thus the
stopping time is finite almost surely.

Let $S$ be the finite queried set. In the second stopping case let $A$
be the discovered open set, so $|A|=n$. Every revealed open vertex
belongs to the explored component, and every revealed closed vertex,
when the root is open, is adjacent to it. The queries reveal the full
uniform values; the exploration decisions use their $b$-open indicators.

Take a second independent uniform field $U''$, and form a mixed field
equal to $U'$ on $S$ and $U''$ on $V\setminus S$. This field has the
original product uniform law. Indeed, conditional on a finite stopped
exploration transcript, all unqueried coordinates of $U'$ remain
independent uniform variables. Replacing them by independent variables
from $U''$ preserves that conditional law. 

Let $C=C_{H_b}(x)$ be the component of $H_b$ containing $x$ under
the mixed field, with $C=\varnothing$ if $x$ is $b$-closed. If the
original exploration exhausted a finite component, its revealed closed
boundary still prevents further growth in the mixed configuration.
Consequently, if $C$ is bad, the second stopping case must have occurred.

Suppose that an infinite component of $G_a$ in the independent field $U''$
meets $A$. By local finiteness, there is a simple infinite $a$-open
$G$-ray from a vertex of that intersection. Let $z$ be its last vertex
in $S$, and let $y$ be the next vertex. The infinite tail starting at
$y$ lies outside $S$, and is therefore still $a$-open in the mixed
configuration. In particular, $y$ belongs to an infinite component of $G_a$
in that configuration.

If $z$ is $b$-open in the mixed field, then $z\in A\subseteq C$, and
the edge $zy\in E(G)\subseteq E(H)$ attaches an infinite component of $G_a$
to $C$, contrary to badness. If $z$ is $b$-closed, it is a revealed
boundary vertex adjacent to $A\subseteq C$, and is $G$-adjacent to
$y\in I_a$. Thus $z\in W(C)$. Hence a bad $C$ with empty $W(C)$ can
occur only if $A$ misses all infinite components of $G_a$ in $U''$.

Conditional on the stopped transcript in the second case, $A$ is a
fixed set of size $n$, independent of $U''$. It follows from
\eqref{rev:uniform-hit-assumption} that
\[
 \PP\bigl(C_{H_b}(x)\text{ is bad and }W(C_{H_b}(x))=\varnothing\bigr)
 \le\epsilon(n).
\]
Let $n\to\infty$, and then take the countable union over $x$.

\medskip
\noindent\emph{Step 2: no bad component has finite $W(C)$.}
Suppose the contrary. Fix an origin $o$. Since every finite subset of
$V$ is contained in some $H$-ball about $o$, there is a deterministic
integer $M\ge0$ for which the event
\[
 F_M=\{\text{some bad component }C\text{ satisfies }
                      W(C)\subseteq B_H(o,M)\}
\]
has positive probability. Put $K=B_H(o,M+1)$, which is finite.
Take independent full uniform fields $U,Z$, and define
$\widehat U=Z$ on $K$ and $\widehat U=U$ off $K$. The field
$\widehat U$ has the original product law, and the event
\[
 F_M(U)\cap\{Z(v)>b\text{ for every }v\in K\}
\]
has probability $\PP(F_M)(1-b)^{|K|}>0$.

On this event, all vertices of $K$ are closed at both levels,
and all vertices outside $K$ keep their original states.
Choose an original bad component $C$ as in $F_M$.
The set $C\setminus K$ is still infinite and has only finitely
many components. Indeed, if $K$ meets $C$, each remaining
component has a neighbour in $C\cap K$, and this finite set has
only finitely many neighbours by local finiteness. If $K$
misses $C$, then $C\setminus K=C$.
Thus $C\setminus K$ has an infinite component $C'$.

In the modified field, the open vertices of $C$ are exactly
$C\setminus K$. Since no vertices were opened, no new open
connection to vertices outside $C$ is possible. Hence $C'$
is a component of the modified $H_b$.
Every infinite $G_a$-component after the modification is
contained in an original infinite $G_a$-component, so
$\widehat I_a\subseteq I_a$.
Since $C\cap I_a=\varnothing$ and $C'\subseteq C$, we have
$C'\cap\widehat I_a=\varnothing$. Thus $C'$ remains bad.

Suppose that $z$ belonged to the new $W(C')$. Then $z$ is $H$-adjacent
to $C'\subseteq C$ and $G$-adjacent to $\widehat I_a\subseteq I_a$.
If it had been $b$-open originally, it would have belonged to $C$, and
its edge to $I_a$ would imply that $C$ contained an infinite component of $G_a$.
This contradicts badness. Thus $z$ was originally $b$-closed and lay in
the old $W(C)\subseteq B_H(o,M)$. Every $H$-neighbour of such a vertex
lies in $K$, so it cannot have a neighbour in $C'\subseteq V\setminus K$.
The new $W(C')$ is therefore empty. The modification produces a bad
component with empty $W$ with positive probability under the original
product law, contradicting Step 1.

\medskip
\noindent\emph{Step 3: no bad component has infinite $W(C)$.}
Represent the two coupled levels by
\begin{equation}\label{rev:sprinkling-representation}
 A(v)\sim\operatorname{Bernoulli}(a),\qquad
 B(v)=A(v)\vee Y(v),\qquad
 \rho=\frac{b-a}{1-a},
\end{equation}
where the $A(v)$ are independent and the $Y(v)$ are independent
Bernoulli($\rho$) variables, independent of the entire field $A$.
The pair of fields $(A,B)$ has the same joint law as the original
coupled configurations at levels $a$ and $b$.

Condition on the whole field $A$. The set $I_a$ and hence
$D=\{z:z\sim_G I_a\}$ are now deterministic. Fix a root $x$ and
fairly explore its component in $H_B$, processing vertices one at a time.
An $A$-open vertex is already known to be $B$-open; at an $A$-closed
vertex query its previously unqueried $Y$-value. Count the queries
whose vertices belong to $D\cap\{A=0\}$. Membership in this set is
determined by the fixed $A$-field. Each counted query is selected from
that field and the preceding exploration history before its own
$Y$-value is seen. Its conditional success probability is exactly
$\rho>0$.

Successive conditioning therefore gives, for every integer $k\ge1$,
\begin{equation}\label{rev:predictable-trials}
 \PP\left(
 \begin{array}{c}
 \text{at least }k\text{ counted queries occur, and}\\
 \text{the first }k\text{ counted queries return zero}
 \end{array}
 \,\middle|\,A\right)
 \le(1-\rho)^k.
\end{equation}

On the event that the explored component $C$ is bad and $W(C)$ is infinite,
every vertex of $W(C)$ is $B$-closed, has $A=0$, and belongs to $D$.
Fairness eventually processes each of these boundary vertices, so
infinitely many counted queries occur. Every counted query must fail:
a counted success would attach a $B$-open vertex adjacent to $I_a$ to
the explored component and would put an infinite component of $G_a$ inside it.
Thus the bad event is contained in the event in
\eqref{rev:predictable-trials} for every $k$. Letting $k\to\infty$
gives conditional probability zero. Integrating over $A$ and taking
the countable union over roots rules out all bad components with infinite
$W$. Together with Steps 1 and 2, this completes the proof.
\end{proof}

For a graph $G$ satisfying the hypotheses of Theorem~\ref{mt1},
Lemma~\ref{rev:uniform-hitting} and~\eqref{rev:hp-estimates} verify
\eqref{rev:uniform-hit-assumption} at every $a>a_0$.
Taking $H=G$ in Theorem~\ref{rev:sprinkled-stability} also gives
stability between two parameters on the same graph. No containment
assertion at equal parameters is needed.

\begin{lemma}\label{rev:near-upper}
Let $X$ be infinite, connected, locally finite and properly embedded
in the plane, with minimum degree at least seven and all finite faces
triangular. For every fixed
\begin{equation}\label{rev:near-upper-interval}
                       1-q_0<b<1-\pcs(X),
\end{equation}
almost surely $X_b$ has infinitely many infinite components.
\end{lemma}

\begin{proof}
Put $q=1-b$, so $\pcs(X)<q<q_0$.
Lemma~\ref{rev:polygon-count} and Theorem~\ref{rev:infinite-ends}
imply that some infinite $q$-open cluster has infinitely many ends
almost surely. Equivalently, in Bernoulli($b$) percolation some infinite
closed cluster has infinitely many ends. Since $b>1-q_0>1/2>a_0$,
Corollary~\ref{rev:finite-deletion}, applied to $X$, supplies an infinite
open cluster in every infinite component after every finite vertex
deletion, simultaneously. Proposition~\ref{rev:high-density} now gives
infinitely many infinite open clusters. By \eqref{rev:hp-estimates}, $\pcs(X)\le h_0<q_0$,
so the interval in \eqref{rev:near-upper-interval} is nonempty.
\end{proof}

\begin{proof}[Proof of Theorem~\ref{mt1}]
The strict bound $\pcs(G)<1/2$ follows from
\eqref{rev:hp-estimates} or, alternatively, from the tree
embedding in Lemma~\ref{l59}. Proposition~\ref{rev:lower-interval} proves
nonuniqueness throughout $(\pcs(G),q_0)$.

Let $H$ be the finite-face triangulation supplied by
Lemma~\ref{rev:augmentation}. Fix
$a_0<a<1-\pcs(G)$. Edge monotonicity gives
$\pcs(H)\le\pcs(G)$. Choose a deterministic $b$ such that
\[
                   \max\{a,1-q_0\}<b<1-\pcs(H).
\]
Such a choice is possible: the first lower bound is strictly below
$1-\pcs(H)$ by the choice of $a$, and the second is strictly below it
because $\pcs(H)\le h_0<q_0$.
Lemma~\ref{rev:near-upper} gives infinitely many infinite components
of $H_b$ almost surely. \mainrevision{Lemma~\ref{rev:uniform-hitting} and
\eqref{rev:hp-estimates} verify~\eqref{rev:uniform-hit-assumption}
for $G_a$. Theorem~\ref{rev:sprinkled-stability} therefore puts an
infinite component of $G_a$ inside each of those components.}
Distinct components of $H_b$ are disjoint, so the contained components
of $G_a$ are distinct. Therefore $G_a$ has infinitely many infinite
components almost surely throughout $(a_0,1-\pcs(G))$.

Finally,
\[
                a_0=0.381966\ldots<q_0=0.384900\ldots.
\]
The two proved intervals overlap, and their union contains
$(\pcs(G),1-\pcs(G))$. This proves the full nonuniqueness assertion.
\end{proof}
\end{revision}

\section*{Acknowledgements}

The author was supported by the Simons Foundation under grant
SFI-MPS-TSM-00025603.

\section*{AI usage disclosure}

GPT-6 assisted with revising the manuscript into its present form
and with formal verification in Lean. The author takes full
responsibility for the contents of the paper.



\begin{thebibliography}{10}

\begin{revision}
\bibitem{bs96}
I.~Benjamini and O.~Schramm.
\newblock Percolation beyond {\(\mathbb{Z}^d\)}, many questions and a few
  answers.
\newblock {\em Electron. Commun. Probab.}, 1:71--82, 1996.
\end{revision}

\bibitem{bsjams}
I.~Benjamini and O.~Schramm.
\newblock Percolation in the hyperbolic plane.
\newblock {\em J. Amer. Math. Soc.}, 14(2):487--507, 2001.

\begin{revision}
\bibitem{BCP}
Nathan Bowler, Johannes Carmesin, and Julian Pott.
\newblock Edge-disjoint double rays in infinite graphs: a {Halin} type result.
\newblock {\em Journal of Combinatorial Theory, Series B}, 111:1--16, 2015.
\newblock Definition 2 and Lemma 3; arXiv:1307.0992v4.
\end{revision}

\bibitem{BK89}
R.~M. Burton and M.~Keane.
\newblock Density and uniqueness in percolation.
\newblock {\em Commun. Math. Phys.}, 121:501--505, 1989.

\bibitem{DCT15}
H.~Duminil-Copin and V.~Tassion.
\newblock A new proof of the sharpness of the phase transition for {B}ernoulli
  percolation and the {I}sing model.
\newblock {\em Commun. Math. Phys.}, 343:725--745, 2016.

\bibitem{GHZ25}
A.~Glazman, M.~Harel, and N.~Zelesko.
\newblock Planar percolation and the loop {$O(n)$} model.
\newblock Preprint, arXiv:2508.20917, 2025.

\bibitem{GrZL221}
G.~R. Grimmett and Z.~Li.
\newblock Percolation critical probabilities of matching lattice-pairs.
\newblock {\em Random Struct. Algorithms}, 65:832--856, 2024.
\newblock arXiv:2205.02734.

\bibitem{GrZL22}
G.~R. Grimmett and Z.~Li.
\newblock Hyperbolic site percolation.
\newblock {\em Random Struct. Algorithms}, 66:e21262, 2025.
\newblock arXiv:2203.00981.

\bibitem{HP19}
J.~Haslegrave and C.~Panagiotis.
\newblock Site percolation and isoperimetric inequalities for plane graphs.
\newblock {\em Random Struct. Algorithms}, 58(1):150--163, 2021.
\newblock arXiv:1905.09723.

\bibitem{Kesten02}
H.~Kesten.
\newblock Some highlights of percolation.
\newblock In \emph{Proceedings of the International Congress of
Mathematicians (Beijing, 2002)}, Vol.~I, pp.~345--362, 2002.

\bibitem{ZL17}
Z.~Li.
\newblock Constrained percolation, {I}sing model, and {XOR} {I}sing model on
  planar lattices.
\newblock {\em Random Struct. Algorithms}, 57(2):474--525, 2020.

\bibitem{perc24}
Zhongyang Li.
\newblock Critical site percolation and cutsets.
\newblock {\em Electron. Commun. Probab.}, 30:1--9, 2025.
\newblock arXiv:2410.22152.

\bibitem{Li26Ends}
Z.~Li.
\newblock Planar site percolation, end structure, and the
  {Benjamini--Schramm} conjecture.
\newblock Preprint, arXiv:2601.09958, 2026.

\bibitem{Li26Packing}
Z.~Li.
\newblock Recursive packing bounds for supercritical disconnection
  in {Bernoulli} site percolation.
\newblock Preprint, arXiv:2601.09950v2, 2026.

\bibitem{ry90}
R.~Lyons.
\newblock Random walk and percolation on trees.
\newblock {\em Ann. Probab.}, 18(3):931--958, 1990.

\bibitem{LP16}
R.~Lyons and Y.~Peres.
\newblock {\em Probability on Trees and Networks}, volume~42 of {\em Cambridge
  Series in Statistical and Probabilistic Mathematics}.
\newblock Cambridge University Press, Cambridge, 2016.
\newblock Available at \url{https://rdlyons.pages.iu.edu/}.

\bibitem{cp19}
C.~Panagiotis.
\newblock Self-avoiding walks and polygons on hyperbolic graphs.
\newblock {\em J. Graph Theory}, 106(3):435--473, 2024.
\newblock arXiv:1908.00127.

\begin{mainrevisionblock}
\bibitem{Schonmann99}
R.~H. Schonmann.
\newblock Stability of infinite clusters in supercritical percolation.
\newblock {\em Probab. Theory Related Fields}, 113(2):287--300, 1999.
\end{mainrevisionblock}

\bibitem{WW98}
W.~Woess.
\newblock A note on tilings and strong isoperimetric inequality.
\newblock {\em Math. Proc. Cambridge Philos. Soc.}, 124(3):385--393, 1998.

\bibitem{ZA97}
A.~{\.Z}uk.
\newblock On the norms of random walks on planar graphs.
\newblock {\em Ann. Inst. Fourier}, 47(5):1463--1490, 1997.

\end{thebibliography}
\end{document}